\documentclass[12pt,draftcls,onecolumn]{IEEEtran}  

\usepackage{epsfig} 
\usepackage{amsmath} 
\usepackage{amssymb}  
\usepackage{subfigure}
\usepackage{color}
\usepackage{dsfont}

\numberwithin{equation}{section}
\newtheorem{lemma}{Lemma}

\newtheorem{theorem}{Theorem}
\newtheorem{corollary}{Corollary}
\newtheorem{remark}{Remark}

\newtheorem{example}{Example}

\begin{document}

\date{}


\begin{center}
\LARGE \bf Content
\end{center}

\vspace{0.5cm}

\noindent This report includes the original manuscript (pp.~2--31) and the supplementary material (pp.~32--38) of ``Realization of Biquadratic Impedance as Five-Element Bridge Networks''.

\vspace{1cm}

\noindent Authors: Michael Z. Q. Chen, Kai Wang, Chanying Li, and Guanrong Chen

\newpage

\title{\LARGE \bf Realization of Biquadratic Impedances as Five-Element Bridge Networks}
\date{}

\author{Michael Z. Q. Chen,~\IEEEmembership{Senior Member,~IEEE,}
        Kai Wang,\\
        Chanying Li,~\IEEEmembership{Member,~IEEE,}
        and~Guanrong Chen,~\IEEEmembership{Fellow,~IEEE}
\thanks{M.~Z.~Q. Chen and K. Wang are with Department of Mechanical Engineering, The University of Hong Kong, Pokfulam
Road, Hong Kong.}
\thanks{C. Li is with
Academy of Mathematics and Systems Science, Chinese Academy of Sciences, Beijing, P. R. China.}
\thanks{G. Chen is with Department of Electronic Engineering, City University of Hong Kong, Tat Chee Avenue, Kowloon, Hong Kong.}
\thanks{Correspondence: MZQ Chen, mzqchen@hku.hk.}
\thanks{This research was partially supported by the Research Grants
Council, Hong Kong, through the General Research Fund
under Grant 17200914,  Grant CityU 11201414,
and the Natural Science Foundation of China under Grant 61374053 and 61422308.}
}


\maketitle

\begin{abstract}
This paper studies the passive network synthesis problem of biquadratic impedances as   five-element bridge networks.
By investigating the realizability conditions of the configurations that can cover all the possible cases, a necessary and sufficient condition is obtained for   biquadratic impedances to be realizable as   two-reactive five-element bridge networks. Based on the types of the two reactive elements, the discussion is divided into two parts, and a canonical form for biquadratic impedances is utilized to simplify and combine the conditions. Moreover, the realizability result for the biquadratic impedance is extended to the general five-element bridge networks. Some numerical examples are presented for illustration.

\medskip

\end{abstract}


\section{Introduction} \label{sec: introduction}

Passive network synthesis has been one of the most important subjects in circuit and system theories. This field has experienced a ``golden era'' from the 1930s to the 1970s \cite{AV73,BD49,Gui57,SR61}.
The most classical transformerless realization method, the Bott-Duffin procedure \cite{BD49}, shows that any positive-real impedance (resp. admittance) is realizable with a finite number of resistors, capacitors, and inductors. However, the resulting networks typically contain a large number of redundant elements, and the minimal realization problem is far from being solved even today for the biquadratic impedance.
Recently, interest in investigation on passive networks has revived \cite{CPSWS09,CS09,CS09(2),CWSL13,CWZC15,JS11,JS12,Smi02,WC12,WCH14,WC15,YKKP14}, due to its connection with passive mechanical control using a new mechanical element named inerter \cite{CHHC14,HCSH15,Smi02}.
In parallel, there are also some new results on the negative imaginary systems \cite{XPL12}.
In addition to the control of mechanical systems, synthesis of passive networks
can also be applied to a series of other fields, such as the microwave antenna circuit design \cite{LBHD11},
filter design \cite{RMB14}, passivity-preserving balanced truncation \cite{RS10}, and biometric image processing \cite{Sae14}. Noticeably, the need for a renewed attempt to passive network synthesis and its contribution to systems theory has been highlighted in \cite{Kal10}.

The realization problem of biquadratic impedances has been a focal  topic in the theory of passive network synthesis \cite{Lad48,Lad64,Rei69,WC12,WCH14}, yet its minimal realization problem has not  been completely  solved to date. In fact, investigation on synthesis of biquadratic impedances can provide significant guidance on  realization  of more general functions. By the Bott-Duffin procedure \cite{BD49} and Pantell's simplification \cite{Pan54}, one needs at most eight elements to realize a general positive-real biquadratic impedance. In \cite{Lad48}, Ladenheim listed $108$ configurations containing
at most five elements and at most two reactive elements that can realize the biquadratic impedance. For each of them, values of the elements are explicitly expressed in terms of the coefficients of the biquadratic function, without any derivation given. Realizability conditions of series-parallel networks are listed, but
those of bridge networks are not.
Furthermore, the realizability problem of biquadratic impedances as five-element networks containing three reactive elements \cite{Lad64}  has been investigated.
Recently, a new concept named regularity is introduced
and applied  to investigate the realization problem of the biquadratic impedances as five-element networks in \cite{JS11}, in which a necessary and sufficient condition  is derived for a biquadratic impedance to be realizable as such a network.
It is noted that only the realizability conditions for five-element bridge networks that are not necessarily equivalent to the corresponding series-parallel ones are investigated in \cite{JS11}. Hence,  necessary and sufficient conditions for the biquadratic impedances to be realizable as   five-element bridge networks are still unknown today.

The present paper is concerned with the realization   of   biquadratic impedances as five-element bridge networks. As discussed above, this problem remains unsolved today.
Pantell's simplification \cite{Pan54} shows that non-series-parallel networks may often contain less redundancy.
Besides, the non-series-parallel structure sometimes has its own advantages in practice \cite{CHD12}. It is essential to construct a  five-element bridge network, the simplest non-series-parallel network to solve the minimal realization problem of biquadratic impedances.
This paper focuses on deriving some realizability conditions of biquadratic impedances as two-reactive five-element bridge networks.
Based on these and some previous results in \cite{Lad64}, the realization result of five-element bridge networks without limiting the number of reactive elements will follow. The discussion on realization of two-reactive five-element bridge networks is divided into two parts, based on whether the two reactive elements are of the same type or not. Through investigating realizability conditions for configurations that can cover all the possible cases, a necessary and sufficient condition is obtained for a biquadratic impedance to be realizable as a two-reactive five-element bridge network. A canonical form for biquadratic impedances is   utilized to simplify and combine the conditions.
Furthermore, the corresponding result of   general five-element bridge networks is further obtained.
Throughout, it is assumed that the given biquadratic impedance is realizable with at least five elements. A part of this paper has appeared as a conference paper in Chinese \cite{CWLC14} (Section V-C and a part of contents in Sections~III and IV).

\section{Preliminaries}  \label{sec: Preliminaries}

A real-rational function $H(s)$ is \textit{positive-real} if $H(s)$ is analytic and $\Re(H(s)) \geq 0$ for $\Re(s) > 0$ \cite{Gui57}. An \textit{impedance} $Z(s)$ is defined as $Z(s) = V(s)/I(s)$, and an \textit{admittance} is $Z^{-1}(s)$, where $V(s)$ and $I(s)$ denote the voltage and current, respectively.
A linear one-port time-invariant network is passive if and only if its impedance (resp. admittance) is positive-real, and any positive-real function is realizable as the impedance (resp. admittance) of  a one-port network  consisting of a finite number of resistors, capacitors, and inductors \cite{BD49,Gui57}, thus the network \textit{realizes} (or being a \textit{realization} of) its impedance (resp. admittance). A \textit{regular} function $H(s)$ is a class of positive-real functions with the smallest value of $\Re(H(j\omega))$ or $\Re(H^{-1}(j\omega))$ being at $\omega = 0 \cup \infty$ \cite{JS11}. The capacitors and inductors are called \textit{reactive elements}, and resistors are called \textit{resistive elements}.
Moreover, the concept of the \textit{network duality} is presented in \cite{CWLC15}.

\section{Problem Formulation}  \label{sec: Problem formulation}

The general form of a \textit{biquadratic impedance} is
\begin{equation}   \label{eq: biquadratic impedance}
Z(s) = \frac{a_2 s^2 + a_1 s + a_0}{b_2 s^2 + b_1 s + b_0},
\end{equation}
where $a_2$, $a_1$, $a_0$, $b_2$, $b_1$, $b_0$ $\geq 0$. It is known from \cite{CS09(2)} that its positive-realness is equivalent to $(\sqrt{a_0b_2}-\sqrt{a_2b_0})^2 \leq a_1b_1$. For brevity, the following notations are introduced:
$\mathds{A} = a_0b_1 - a_1b_0$, $\mathds{B} = a_0b_2 - a_2b_0$, $\mathds{C} = a_1b_2 - a_2b_1$,
$\mathds{D}_a := a_1 \mathds{A} - a_0 \mathds{B}$,
$\mathds{D}_b := -b_1 \mathds{A} + b_0 \mathds{B}$,
$\mathds{E}_a := a_2 \mathds{B} - a_1 \mathds{C}$,
$\mathds{E}_b := -b_2 \mathds{B} + b_1 \mathds{C}$,
$\mathds{M} := a_0 b_2 + a_2 b_0$,
$\Delta_a  := a_1^2-4a_0a_2$,
$\Delta_b  := b_1^2-4b_0b_2$,
$\Delta_{ab}  := a_1b_1 - 2\mathds{M}$,
$\mathds{R} := \mathds{A} \mathds{C} - \mathds{B}^2$,
$\Gamma_a   := \mathds{R} + b_0b_2 \Delta_a$, and
$\Gamma_b  :=  \mathds{R} + a_0a_2 \Delta_b$.



As shown in \cite{JS11},  if at least one of  $a_2$, $a_1$, $a_0$, $b_2$, $b_1$, and $b_0$ is zero, then $Z(s)$ is realizable with at most two reactive elements and two resistors.
In \cite{WCH14}, a necessary and sufficient condition for a biquadratic impedance with positive coefficients to be realizable with at most four elements was established, as below.
\begin{lemma} \cite{WCH14}  \label{lemma: condition of at most four}
{A biquadratic impedance $Z(s)$ in the form of \eqref{eq: biquadratic
impedance}, where $a_2$, $a_1$, $a_0$, $b_2$, $b_1$, $b_0$ $> 0$, can be realized with at most four elements if and only if at least one of the following conditions  holds: 1) $\mathds{R} = 0$; 2) $\mathds{B} = 0$; 3) $\mathds{B} > 0$ and $\mathds{D}_a \mathds{E}_b = 0$; 4) $\mathds{B} < 0$ and $\mathds{D}_b \mathds{E}_a = 0$; 5) $\Gamma_a \Gamma_b = 0$. }
\end{lemma}

Therefore, when investigating the realizability problem of five-element networks, it suffices to assume that $a_2$, $a_1$, $a_0$, $b_2$, $b_1$, $b_0$ $> 0$ but the condition of Lemma~\ref{lemma: condition of at most four} does not hold in the consideration of minimal realizations. For brevity, the set of all such biquadratic functions (with positive coefficients and with the condition of Lemma~\ref{lemma: condition of at most four} not being  satisfied) is denoted by $\mathcal{Z}_b$ in this paper.

The present paper aims to  derive a necessary and sufficient condition for a biquadratic impedance $Z(s) \in \mathcal{Z}_b$ to be realizable as a two-reactive five-element bridge network (Theorem~\ref{theorem: final condition}) and the corresponding result for a general five-element bridge network (Theorem~\ref{theorem: final corollary}). Figs.~\ref{fig: Quartet-01-same-kind}--\ref{fig: Quartet-03-different-kind} and \ref{fig: Three-Reactive-Quartet} are the corresponding realizations.
The configurations are assumed to be passive one-port  time-invariant transformerless networks containing at most three kinds of passive elements, which are resistors, capacitors, and inductors, and the values of the elements are all positive and finite.

\section{A Canonical Biquadratic Form}
\label{sec: A Canonical Biquadratic Form}

A canonical form $Z_c(s)$ for biquadratic impedances stated in is expressed as
\begin{equation}  \label{eq: canonical form}
Z_c(s) = \frac{s^2 + 2U\sqrt{W}s + W}{s^2 + (2V/\sqrt{W})s + 1/W},
\end{equation}
where
\begin{equation}  \label{eq: from Zc to Z}
W = \sqrt{\frac{a_0b_2}{a_2b_0}}, ~~~ U = \frac{a_1}{2\sqrt{a_0a_2}}, ~~~ V = \frac{b_1}{2\sqrt{b_0b_2}}.
\end{equation}
It is not difficult to verify that $Z_c(s)$ can be obtained from $Z(s)$ through
$Z_c(s) = \alpha Z(\beta s)$, where $\alpha = b_2/a_2$ and
$\beta = \sqrt[4]{a_0b_0/(a_2b_2)}$. If $Z(s)$ is realizable as a network $N$, then the corresponding $Z_c(s)$ must be realizable as another network $N_c$ with the same one-terminal-pair labeled graph by a proper transformation of the element values, and \textit{vice versa}.
Therefore, the realizability condition for $Z_c(s)$ as a network whose one-terminal-pair labeled graph is $\mathcal{N}$ in terms of $U$, $V$, $W$ $> 0$ can be determined from that of $Z(s)$ in terms of $a_2$, $a_1$, $a_0$, $b_2$, $b_1$, $b_0$ $> 0$, via transformation
\begin{equation}  \label{eq: from Z to Zc}
\begin{split}
a_2 = 1, \ a_1 = 2U\sqrt{W},  \ a_0 = W,   \\
b_2 = 1, \ b_1 = 2V/\sqrt{W}, \ b_0  = 1/W.
\end{split}
\end{equation}
Conversely, the realizability condition for $Z(s)$ as a network with one-terminal-pair labeled graph  $\mathcal{N}$ in terms of $a_2$, $a_1$, $a_0$, $b_2$, $b_1$, $b_0$ $> 0$ can be determined from that for $Z_c(s)$ in terms of $U$, $V$, $W$ $> 0$, via transformation \eqref{eq: from Zc to Z}.
Furthermore, through \eqref{eq: from Z to Zc}, one concludes that $Z_c(s)$ is positive-real if and only if $\sigma_c := 4UV + 2 - (W + W^{-1}) \geq 0$, as stated in \cite{JS11}. Notations $\Delta_{ab}$, $\mathds{R}$, $\Gamma_a$, and $\Gamma_b$, as defined in Section~\ref{sec: Problem formulation} are respectively converted to
$\Delta_{ab_c} := 4 U V - 2 (W + W^{-1})$,
$\mathds{R}_c := -4 U^2 - 4 V^2 + 4 U V (W + W^{-1}) - (W - W^{-1})^2$, $\Gamma_{a_c} := - 4 V^2 + 4 U V (W + W^{-1}) - (W + W^{-1})^2$, and $\Gamma_{b_c} := - 4 U^2 + 4 U V (W + W^{-1}) - (W + W^{-1})^2$.  Also, $\mathds{M} \mathds{R} + 2a_0a_2b_0b_2\Delta_{ab}$ is converted to $-(W + W^{-1})^3 + 4 U V (W + W^{-1})^2 - 4(U^2 + V^2) (W + W^{-1}) + 8 U V$.
Moreover, for brevity, denote $\lambda_c := 4 U V - 4 V^2 W + (W - W^{-1})$.
Defining $\rho^{\ast}(U,V,W)=\rho(U,V,W^{-1})$ and $\rho^{\dag}(U,V,W)=\rho(V,U,W)$ for any rational function $\rho(U,V,W)$, one can see that
$\lambda_c^{\ast\dag}W$, $\lambda_c/W$, $\lambda_c^{\dag}$, and $\lambda_c^{\ast}$
correspond to
$\mathds{D}_a$, $\mathds{D}_b$, $\mathds{E}_a$, $\mathds{E}_b$,
respectively, through
\eqref{eq: from Z to Zc}. Besides, by denoting  $\eta_{c} := 4 U^2 + 4 V^2 + 4 U V (3 W - W^{-1}) + (W - W^{-1})(9W - W^{-1})$ and $\zeta_{c} :=-4 U^2 - 4 V^2 + 4 U V (W + W^{-1}) - (W - W^{-1})(3W - W^{-1})$,  corresponding to $(-\mathds{R} + 4a_0b_2(a_1b_1+2\mathds{B}))$ and $(\mathds{R} - 2a_0b_2\mathds{B})$, respectively,  one has $\eta_{c}^{\ast} = \eta_{c}^{\ast\dag}$ and $\zeta_{c}^{\ast} = \zeta_{c}^{\ast\dag}$ corresponding to $(-\mathds{R} + 4a_2b_0(a_1b_1 - 2\mathds{B}))$ and $(\mathds{R} + 2a_2b_0\mathds{B})$, respectively.

Denote $\mathcal{Z}_{b_c}$ as the set of biquadratic functions in the form of \eqref{eq: canonical form}, where the coefficients $U$, $V$, $W$ $> 0$
and they do not satisfy the condition of Lemma~\ref{lemma: condition of at most four} transformed through \eqref{eq: from Z to Zc}. It is clear that $Z(s) \in \mathcal{Z}_b$ if and only if $Z_c(s) \in \mathcal{Z}_{b_c}$.

In this paper, the canonical biquadratic form as in
\eqref{eq: canonical form} is introduced so as to further simplify the realizability conditions
of \eqref{eq: biquadratic impedance} (in the proof of Theorems~2, 4, and 6).

\section{Main Results}  \label{sec: main results}

Section~\ref{subsec: Preliminaries} presents some basic lemmas that will be used in the following discussions. Section~\ref{subsec: same type} investigates the realization   of   biquadratic impedances as a five-element bridge network containing two reactive elements of the same type.
In Section~\ref{subsec: different type}, the realization problem of biquadratic impedances as a five-element bridge network containing one
inductor and one capacitor is investigated.
Section~\ref{subsec: summary and notes} presents the final results (Theorems~\ref{theorem: final condition} and \ref{theorem: final corollary}).

\subsection{Basic Lemmas} \label{subsec: Preliminaries}


\begin{lemma} \cite{HS12}   \label{lemma: type of elements}
{If a biquadratic impedance $Z(s) \in \mathcal{Z}_b$ is realizable
with two reactive elements of different types and an arbitrary number of resistors, then $\mathds{R} < 0$. If a biquadratic impedance $Z(s) \in \mathcal{Z}_b$ is realizable with two reactive elements of the same type and an arbitrary number of resistors, then $\mathds{R} > 0$.}
\end{lemma}


Let $\mathcal{P}(a,a')$ denote the \textit{path} (see \cite[pg.~14]{SR61}) whose \textit{terminal vertices} (see \cite[pg.~14]{SR61}) are $a$ and $a'$ \cite{CWSL13}; let $\mathcal{C}(a,a')$ denote the \textit{cut-set} (see \cite[pg.~28]{SR61}) that separates $\mathcal{N}$ into two connected subgraphs $\mathcal{N}_1$ and $\mathcal{N}_2$ containing $a$ and $a'$, respectively \cite{CWSL13}.

\begin{lemma} \cite{WCH14}  \label{lemma: topological structure}
{For a network with two terminals $a$ and $a'$ that realizes a biquadratic impedance $Z(s) \in \mathcal{Z}_b$, its network graph can neither contain  the  path $\mathcal{P}(a,a')$ nor contain the cut-set $\mathcal{C}(a,a')$ whose edges correspond to only one kind of reactive elements.}
\end{lemma}

\begin{lemma}   \label{lemma: U V W lemma}
{If $U$, $V$, $W$ $> 0$ satisfy $W \neq 3$,
\begin{equation}
- 4U^2 - 4V^2 + 4UV(W +  W^{-1})-(W -  W^{-1})^2 > 0,
\label{eq: lemma U V W condition 01}
\end{equation}
and
\begin{equation}
\begin{split}
-4U^2 - 4V^2 &+ 4UV(W -  3W^{-1}) \\
&- (W -  W^{-1})(W - 9W^{-1}) \geq 0,
\label{eq: lemma U V W condition 02}
\end{split}
\end{equation}
then
\begin{equation}  \label{eq: lemma U V W condition 03}
4\zeta_{c}^{\ast}W^{-1}(UW-V)(UW-3V) + 8\lambda_c^{\ast}(V^2 - U^2) > 0.
\end{equation}}
\end{lemma}
\begin{proof}
See \cite{CWLC15} for   details.
\end{proof}



\subsection{Five-Element Bridge Networks with Two Reactive Elements of the Same Type}  \label{subsec: same type}

\begin{lemma}  \label{lemma: realization of Five-Element Bridge Networks with Two Reactive Elements of the Same Type}
{A biquadratic impedance $Z(s) \in \mathcal{Z}_b$ is realizable as a five-element bridge network containing two reactive elements of the same type if and only if $Z(s)$ is the impedance of one of   configurations in Figs.~\ref{fig: Quartet-01-same-kind} and \ref{fig: Quartet-02-same-kind}.}
\end{lemma}
\begin{proof}
The proof is straightforward based on
Lemma~\ref{lemma: topological structure}   using the method of enumeration.
\end{proof}

\begin{figure}[thpb]
      \centering
      \includegraphics[scale=1.1]{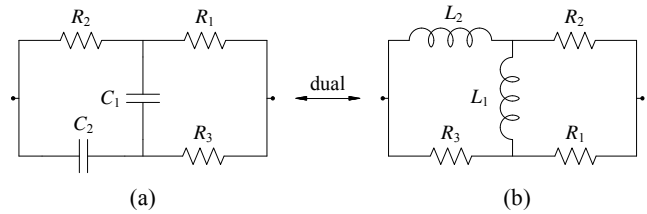}
      \caption{The two-reactive five-element bridge configurations containing the same type of reactive elements, which are  respectively supported by two
      one-terminal-pair labeled graphs $\mathcal{N}_1$ and $\text{Dual}(\mathcal{N}_1)$,  where (a) is No.~85 configuration in \cite{Lad48} and (b) is No.~60 configuration in \cite{Lad48}.}
      \label{fig: Quartet-01-same-kind}
\end{figure}

\begin{figure}[thpb]
      \centering
      \includegraphics[scale=1.1]{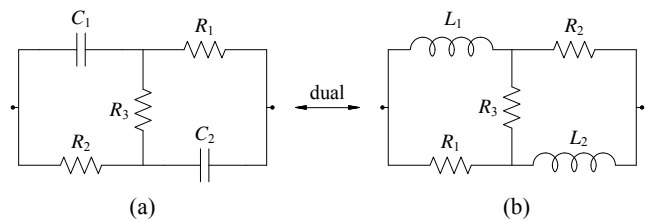}
      \caption{The two-reactive five-element bridge configurations containing the same type of reactive elements, which are respectively supported by two
      one-terminal-pair labeled graphs $\mathcal{N}_2$ and $\text{Dual}(\mathcal{N}_2)$, where (a) is No.~86 configuration in \cite{Lad48} and (b) is No.~61 configuration in \cite{Lad48}.}
      \label{fig: Quartet-02-same-kind}
\end{figure}


\begin{theorem}
\label{theorem: condition of the No. 85 network}
{A biquadratic impedance $Z(s) \in \mathcal{Z}_b$ is realizable as the configuration in
Fig.~\ref{fig: Quartet-01-same-kind}  if and only if
$\mathds{R} - 4a_0a_2b_0b_2 \geq 0$. Furthermore, if $\mathds{R} - 4a_0a_2b_0b_2 \geq 0$ and $\mathds{B} > 0$,
then $Z(s)$ is realizable as the configuration in Fig.~\ref{fig: Quartet-01-same-kind}(a) with values of elements satisfying
\begin{subequations}
\begin{align}
R_2 &= \frac{a_0 - b_0R_1}{b_0}, \label{eq: No. 85 values R2}  \\
R_3 &= \frac{a_2 R_1}{b_2 R_1 - a_2},  \label{eq: No. 85 values R3}       \\
C_1 &= \frac{(a_1 a_2b_0 + a_0 \mathds{C}) R_1 - a_0a_1a_2}{(a_0-b_0R_1)R_1^2 \mathds{M}},
\label{eq: No. 85 values C1}    \\
C_2 &= \frac{(a_2b_0b_1 + b_2 \mathds{A}) - b_0b_1b_2 R_1}{(a_0-b_0R_1) \mathds{M}},
\label{eq: No. 85 values C2}
\end{align}
\end{subequations}
and $R_1$ is the positive root of the following quadratic equation:
\begin{equation}  \label{eq: equation for the root No. 85}
b_0 b_2 \Gamma_a  R_1^2 - (\mathds{M}\mathds{R} + 2a_0a_2b_0b_2\Delta_{ab}) R_1 + a_0a_2 \Gamma_b = 0.
\end{equation}
}
\end{theorem}
\begin{proof}
\textit{Necessity.}
The impedance of the configuration in Fig.~\ref{fig: Quartet-01-same-kind}(a) is
\begin{equation}   \label{eq: general impedance of No. 85 network}
Z(s) = \frac{a(s)}{b(s)},
\end{equation}
where $a(s) = R_1R_2R_3C_1C_2s^2 + ((R_1R_2+R_2R_3+R_1R_3)C_1+(R_1+R_2)R_3C_2)s+(R_1+R_2)$ and $b(s) = (R_1+R_3)R_2C_1C_2s^2 + ((R_1+R_3)C_1+(R_1+R_2+R_3)C_2)s + 1$.
Supposing that $Z(s) \in \mathcal{Z}_b$ is realizable as the configuration in
Fig.~\ref{fig: Quartet-01-same-kind}(a), it follows that
\begin{subequations}
\begin{align}
R_1 R_2 R_3 C_1 C_2  &= ka_2,  \label{eq: No. 85 ka2}  \\
(R_1R_2+ R_3(R_1 + R_2))C_1+(R_1+R_2)R_3&C_2  = ka_1,
\label{eq: No. 85 ka1}   \\
R_1+R_2 &= ka_0,  \label{eq: No. 85 ka0} \\
(R_1+R_3)R_2 C_1C_2 &= kb_2, \label{eq: No. 85 kb2} \\
(R_1+R_3)C_1+(R_1+R_2+R_3)C_2 &= kb_1,  \label{eq: No. 85 kb1} \\
1 &= kb_0.  \label{eq: No. 85 kb0}
\end{align}
\end{subequations}
From \eqref{eq: No. 85 kb0}, one obtains
\begin{equation}  \label{eq: No. 85 k}
k = \frac{1}{b_0}.
\end{equation}
From \eqref{eq: No. 85 ka2} and \eqref{eq: No. 85 kb2}, it follows that
$1/R_1 + 1/R_3 = b_2/a_2$.
The assumption that $R_1 > 0$ and $R_3 > 0$ implies
\begin{equation}  \label{eq: No. 85 positive restriction 02}
b_2 R_1 - a_2 > 0.
\end{equation}
Hence, $R_3$ is solved as \eqref{eq: No. 85 values R3}. Based on \eqref{eq: No. 85 ka0} and \eqref{eq: No. 85 k}, $R_2$ is solved as \eqref{eq: No. 85 values R2}, implying
\begin{equation}
a_0 - b_0 R_1 > 0.    \label{eq: No. 85 positive restriction 01}
\end{equation}
Substituting \eqref{eq: No. 85 values R2}, \eqref{eq: No. 85 values R3}, and \eqref{eq: No. 85 k} into \eqref{eq: No. 85 ka1} and \eqref{eq: No. 85 kb1}, $C_1$ and $C_2$ can be solved as \eqref{eq: No. 85 values C1} and
\eqref{eq: No. 85 values C2}, implying
\begin{align}
(a_1 a_2b_0 + a_0 \mathds{C}) R_1 - a_0a_1a_2 &> 0,
\label{eq: No. 85 positive restriction 03}   \\
(a_2b_0b_1 + b_2 \mathds{A}) - b_0b_1b_2 R_1 &> 0.
\label{eq: No. 85 positive restriction 04}
\end{align}
Substituting \eqref{eq: No. 85 values R2}--\eqref{eq: No. 85 values C2} and \eqref{eq: No. 85 k} into \eqref{eq: No. 85 ka2} yields  \eqref{eq: equation for the root No. 85}. The discriminant of \eqref{eq: equation for the root No. 85} in $R_1$ is obtained as
\begin{equation}  \label{eq: discriminant of No. 85}
\begin{split}
\delta &= (- \mathds{M}\mathds{R} - 2a_0a_2b_0b_2\Delta_{ab})^2 - 4b_0b_2\Gamma_a a_0a_2\Gamma_b  \\
&= \mathds{M}^2 \mathds{R} (\mathds{R} - 4a_0a_2b_0b_2),
\end{split}
\end{equation}
which must be nonnegative. Together with Lemma~\ref{lemma: type of elements}, it follows that $\mathds{R} - 4a_0a_2b_0b_2 \geq 0$. Moreover, from \eqref{eq: No. 85 positive restriction 02} and \eqref{eq: No. 85 positive restriction 01}, one obtains
$\mathds{B} > 0$. Therefore, if $Z(s) \in \mathcal{Z}_b$ is realizable as the configuration in Fig.~\ref{fig: Quartet-01-same-kind}(b), then $\mathds{R} - 4a_0a_2b_0b_2 \geq 0$ and $\mathds{B} < 0$, which are obtained through the principle of duality ($a_2 \leftrightarrow b_2$, $a_1 \leftrightarrow b_1$, and $a_0 \leftrightarrow b_0$) \cite{CWLC15}.

\textit{Sufficiency.}
By the principle of duality, it suffices to show that if  $\mathds{R} - 4a_0a_2b_0b_2 \geq 0$ and
$\mathds{B} > 0$, then $Z(s) \in \mathcal{Z}_b$ is realizable as the configuration  in Fig.~\ref{fig: Quartet-01-same-kind}(a).
Since $\mathds{R} - 4a_0a_2b_0b_2 \geq 0$, it follows that
$\Gamma_a   \geq  4a_0a_2b_0b_2 + b_0b_2 \Delta_a =  a_1^2 b_0 b_2 > 0$, $\Gamma_b   \geq  4a_0a_2b_0b_2 + a_0a_2 \Delta_b = b_1^2 a_0 a_2 > 0$,
$\mathds{M}\mathds{R} + 2a_0a_2b_0b_2(a_1b_1 - 2 \mathds{M}) \geq 4a_0a_2b_0b_2 \mathds{M} + 2a_0a_2b_0b_2(a_1b_1 - 2 \mathds{M}) = 2a_0a_1a_2b_0b_1b_2 > 0$, and
the discriminant of  \eqref{eq: equation for the root No. 85} in $R_1$ as expressed in \eqref{eq: discriminant of No. 85} is nonnegative. Hence,  \eqref{eq: equation for the root No. 85} in $R_1$ has one or two nonzero real roots, which must be positive.

Moreover, $\mathds{A} \mathds{C} > 0$
since $\mathds{R} = \mathds{A} \mathds{C} - \mathds{B}^2 > 0$.
Assume that $\mathds{A} < 0$, that is,
$a_0 b_1 < a_1 b_0$.
Together with $\mathds{B} > 0$, that is,
$a_2 b_0 < a_0 b_2$,
one obtains that $(a_0 b_1) (a_2 b_0) < (a_1 b_0) (a_0 b_2)$, which is equivalent to
$\mathds{C} > 0$. This contradicts   the fact that $\mathds{A} \mathds{C} > 0$
as derived above. Therefore, it is only possible that
$\mathds{A} > 0$ and $\mathds{C} > 0$,
which implies that
$a_1 a_2 b_0 + a_0 \mathds{C} > 0$ and $a_2b_0b_1 + b_2 \mathds{A} > 0$. Replacing $R_1$ in \eqref{eq: equation for the root No. 85} by $a_0/b_0$ and $a_2/b_2$ yields
$a_0^2 \mathds{C}^2 > 0$ and
$a_2^2 \mathds{A}^2 > 0$, respectively. Therefore, $a_0 - b_0R_1 \neq 0$ and $b_2R_1 - a_2 \neq 0$ provided that $R_1$ is the positive root of  \eqref{eq: equation for the root No. 85}.

For the configuration in Fig.~\ref{fig: Quartet-01-same-kind}(a), let the values of the elements therein satisfy \eqref{eq: No. 85 values R2}--\eqref{eq: No. 85 values C2}, where $R_1$ is the positive root of \eqref{eq: equation for the root No. 85}. Let the value of $k$ satisfy \eqref{eq: No. 85 k}. It can be verified that \eqref{eq: No. 85 ka2}--\eqref{eq: No. 85 kb0} hold. Hence, it follows that
$k^4 \mathds{R} = ((R_1+R_3)R_1C_1 + (R_1+R_2)R_3C_2)^2R_2^2C_1C_2$,
which implies that $C_1$ and $C_2$ must be simultaneously  positive or negative. This means that
$((a_1a_2b_0 + a_0 \mathds{C})R_1 - a_0a_1a_2)$  and
$((a_2b_0b_1 + b_2 \mathds{A}) - b_0b_1b_2R_1)$ are simultaneously positive or negative. Assume that
$(a_1a_2b_0 + a_0 \mathds{C})R_1 - a_0a_1a_2 < 0$ and
$(a_2b_0b_1 + b_2 \mathds{A}) - b_0b_1b_2R_1 < 0$.
Then, one obtains
$(a_1a_2b_0 + a_0 \mathds{C})(a_2b_0b_1 + b_2 \mathds{A}) < a_0a_1a_2 b_0b_1b_2$,
which is equivalent to $\mathds{A} \mathds{C} < 0$.
This contradicts   $\mathds{A} \mathds{C} > 0$. Therefore, conditions~\eqref{eq: No. 85 positive restriction 03} and \eqref{eq: No. 85 positive restriction 04} hold, which is equivalent to
$(a_0a_1a_2)/(a_1a_2b_0+a_0\mathds{C}) < R_1 < (a_2b_0b_1 + b_2 \mathds{A})/(b_0b_1b_2)$. Since $(a_0a_1a_2)/(a_1a_2b_0+a_0\mathds{C}) > a_2/b_2$ and $(a_2b_0b_1 + b_2 \mathds{A})/(b_0b_1b_2) < a_0/b_0$
because of $\mathds{A} > 0$ and $\mathds{C} > 0$, it follows that conditions~\eqref{eq: No. 85 positive restriction 01} and \eqref{eq: No. 85 positive restriction 02} must hold. Hence, the values of elements as expressed in \eqref{eq: No. 85 values R2}--\eqref{eq: equation for the root No. 85} must be positive and finite. As a conclusion, the given impedance $Z(s)$ is realizable as the specified network.
\end{proof}

\begin{lemma}   \label{lemma: condition of the No. 86 network}
{A biquadratic impedance $Z(s) \in \mathcal{Z}_b$ is realizable as the configuration in
Fig.~\ref{fig: Quartet-02-same-kind}(a) with $R_1 \neq R_2$ if and only if there exists a positive root of
\begin{equation}  \label{eq: equation for the root No. 86}
b_0 \mathds{E}_b R_3^2 - (\mathds{R} + 2a_2b_0\mathds{B}) R_3 + a_2 \mathds{D}_a = 0
\end{equation}
in $R_3$ such that
\begin{subequations}
\begin{align}
b_2 R_3 - a_2 &> 0, \label{eq: No. 86 positive restriction 01}  \\
a_0 - b_0 R_3 &> 0, \label{eq: No. 86 positive restriction 02}  \\
b_0 (\mathds{C} - 2a_2b_1)\mathds{C}R_3 + a_2(a_0&a_2b_1^2 - a_1^2b_0b_2) > 0.  \label{eq: No. 86 positive restriction 03}
\end{align}
\end{subequations}
Furthermore, if the condition is satisfied and if $R_1 > R_2$,
then the values of the elements are expressed as
\begin{subequations}
\begin{align}
R_1 &= \frac{(a_0 - b_0R_3)(1+\sqrt{\Lambda})}{2b_0},
\label{eq: No. 86 values R1}  \\
R_2 &= \frac{(a_0 - b_0R_3)(1-\sqrt{\Lambda})}{2b_0},
\label{eq: No. 86 values R2}       \\
C_1 &= \frac{a_1 - b_1 R_2}{b_0(R_2+R_3)(R_1 - R_2)},
\label{eq: No. 86 values C1}    \\
C_2 &= \frac{b_1 R_1 - a_1}{b_0(R_1+R_3)(R_1 - R_2)},
\label{eq: No. 86 values C2}
\end{align}
\end{subequations}
where
\begin{equation}
\Lambda = 1 - \frac{4 a_2 b_0 R_3}{(b_2 R_3 - a_2)(a_0 - b_0R_3)},
\end{equation}
and $R_3$ is the positive root of
\eqref{eq: equation for the root No. 86} satisfying
\eqref{eq: No. 86 positive restriction 01}--\eqref{eq: No. 86 positive restriction 03}.
}
\end{lemma}
\begin{proof}
\textit{Necessity.}
The impedance of the configuration in Fig.~\ref{fig: Quartet-02-same-kind}(a) is obtained as
\begin{equation}  \label{eq: general impedance of No. 86 network}
Z(s) = \frac{a(s)}{b(s)},
\end{equation}
where  $a(s) = R_1R_2R_3C_1C_2 s^2 + ((R_2+R_3)R_1C_1 + (R_1+R_3)R_2C_2)s + (R_1 + R_2 + R_3)$ and $b(s) = (R_1R_2 + R_2R_3 + R_3R_1)C_1C_2s^2 + ((R_2+R_3)C_1 + (R_1+R_3)C_2)s+1$.
Then,
\begin{subequations}
\begin{align}
R_1R_2R_3C_1C_2 &= ka_2,    \label{eq: No. 86 ka2}  \\
(R_2 + R_3)R_1 C_1 + (R_1 + R_3)R_2C_2 &= ka_1, \label{eq: No. 86 ka1}  \\
R_1 + R_2 + R_3 &= ka_0,  \label{eq: No. 86 ka0}  \\
(R_1R_2 + R_2R_3 + R_3R_1)C_1C_2 &= kb_2,  \label{eq: No. 86 kb2}  \\
(R_2+R_3)C_1 + (R_1+R_3)C_2 &= kb_1,  \label{eq: No. 86 kb1}  \\
1 &= kb_0.  \label{eq: No. 86 kb0}
\end{align}
\end{subequations}
It is obvious that \eqref{eq: No. 86 kb0} is equivalent to
\begin{equation}   \label{eq: No. 86 k}
k = \frac{1}{b_0}.
\end{equation}
Together with \eqref{eq: No. 86 ka1} and \eqref{eq: No. 86 kb1}, $C_1$ and $C_2$ are solved as \eqref{eq: No. 86 values C1} and \eqref{eq: No. 86 values C2}.
From \eqref{eq: No. 86 ka2} and \eqref{eq: No. 86 kb2}, one obtains
\begin{equation}  \label{eq: No. 86 1/R1 + 1/R2 + 1/R3}
\frac{1}{R_1} + \frac{1}{R_2} + \frac{1}{R_3} = \frac{b_2}{a_2}.
\end{equation}
As a result, condition~\eqref{eq: No. 86 positive restriction 01} is derived.
Due to the symmetry of this configuration, one can assume
that $R_1 > R_2$ without loss of generality. Therefore, from \eqref{eq: No. 86 ka0}, \eqref{eq: No. 86 k}, and \eqref{eq: No. 86 1/R1 + 1/R2 + 1/R3}, $R_1$ and $R_2$ are solved as \eqref{eq: No. 86 values R1} and \eqref{eq: No. 86 values R2}, which implies
condition~\eqref{eq: No. 86 positive restriction 02}. Substituting the expressions of \eqref{eq: No. 86 values R1}--\eqref{eq: No. 86 values C2} and \eqref{eq: No. 86 k} into $(R_1R_2R_3C_1C_2 - ka_2)$ gives
\begin{equation}
\begin{split}
&R_1R_2R_3C_1C_2 - ka_2 \\
=&  \frac{a_2}{b_0} \cdot
\frac{b_0 \mathds{E}_b R_3^2 - (\mathds{R} + 2a_2b_0\mathds{B}) R_3 + a_2 \mathds{D}_a}{\mathds{B}(b_0 b_2 R_3^2 + (2 a_2 b_0 - \mathds{B}) R_3 + a_0 a_2)}.
\end{split}
\end{equation}
Then, one obtains \eqref{eq: equation for the root No. 86}.
Since $\Lambda$ must be nonnegative and $(b_0 b_2 R_3^2 + (2 a_2 b_0 - \mathds{B}) R_3 + a_0 a_2)$ cannot be zero, it follows that
\begin{equation}  \label{eq: No. 86 negative}
b_0 b_2 R_3^2 + (2 a_2 b_0 - \mathds{B}) R_3 + a_0 a_2 < 0.
\end{equation}
Substituting the expressions of the roots of  \eqref{eq: equation for the root No. 86} into \eqref{eq: No. 86 negative} yields
$((\mathds{R} + 2a_2b_0\mathds{B}) + \sqrt{(\mathds{R} + 2a_2b_0\mathds{B})^2 - 4a_2b_0
\mathds{D}_a\mathds{E}_b})
(\mathds{C} - 2a_2b_1)\mathds{C}  + 2a_2 \mathds{E}_b
(a_0a_2b_1^2 - a_1^2b_0b_2) > 0$ or $((\mathds{R} + 2a_2b_0\mathds{B}) - \sqrt{(\mathds{R} + 2a_2b_0\mathds{B})^2 - 4a_2b_0\mathds{D}_a\mathds{E}_b})
(\mathds{C} - 2a_2b_1)\mathds{C} + 2a_2 \mathds{E}_b
(a_0a_2b_1^2 - a_1^2b_0b_2) > 0$,
which is   equivalent to condition~\eqref{eq: No. 86 positive restriction 03}.

\textit{Sufficiency.}
Let the values of the elements satisfy
\eqref{eq: No. 86 values R1}--\eqref{eq: No. 86 values C2}, and $R_3$ be a positive root of \eqref{eq: equation for the root No. 86} satisfying \eqref{eq: No. 86 positive restriction 01} and \eqref{eq: No. 86 positive restriction 02}. Let $k$ satisfy \eqref{eq: No. 86 k}. Then, it can be verified that \eqref{eq: No. 86 ka2}--\eqref{eq: No. 86 kb0} are satisfied. Now, it suffices to prove that values of elements must be positive and finite.
From the discussion in  the necessity part, it is noted that condition~\eqref{eq: No. 86 positive restriction 03} yields $\Lambda
> 0$, and
conditions~\eqref{eq: No. 86 positive restriction 01} and \eqref{eq: No. 86 positive restriction 02} imply $R_1 > 0$ and $R_2 > 0$. Besides, one has
\begin{equation}
\begin{split}
C_1 C_2
=&  \frac{b_1^2(a_0-b_0R_3)^2\Lambda -
(\mathds{A} - a_1b_0 - b_0b_1R_3)^2}{(a_0 - b_0R_3)^2((a_0+b_0R_3)^2-(a_0-b_0R_3)^2\Lambda)\Lambda}  \\
=& \frac{(b_2 R_3 - a_2)(b_0b_1\mathds{C}R_3^2 -
\mathds{A}\mathds{C}  R_3 + a_1a_2 \mathds{A})}
{(a_0 - b_0R_3)(b_0b_2R_3^2 +
(2a_2b_0 - \mathds{B}) R_3 + a_0a_2)R_3^2\mathds{B}} \\
=& \frac{(b_2 R_3 - a_2)}{(a_0 - b_0R_3)R_3^2} > 0,
\end{split}
\end{equation}
where the third equality is \eqref{eq: equation for the root No. 86}.
Since $R_1 > R_2$, it follows from \eqref{eq: No. 86 values C1} and \eqref{eq: No. 86 values C2} that $C_1 > 0$ and $C_2 > 0$.
\end{proof}

The realizability condition for the configuration in Fig.~\ref{fig: Quartet-02-same-kind}(a) with $R_1 = R_2$ is derived as follows.

\begin{lemma}  \label{lemma: condition of the No. 86' network}
{A biquadratic impedance $Z(s) \in \mathcal{Z}_b$ is realizable as the configuration in Fig.~\ref{fig: Quartet-02-same-kind}(a) with $R_1 = R_2$  if and only if
\begin{subequations}
\begin{align}
a_1b_1^2(2\mathds{A} - a_1b_0) - 4b_2\mathds{A}^2 &\geq 0,
\label{eq: No. 86 network condition 01 R1 = R2}  \\
a_1b_2(\mathds{A} - a_1b_0)
= a_2b_1(2\mathds{A} - a_1b_0) &> 0.
\label{eq: No. 86 network condition 02 R1 = R2}
\end{align}
\end{subequations}
Furthermore, if the condition is satisfied, then the values of the elements are expressed as
\begin{subequations}
\begin{align}
R_1 &= \frac{a_1}{b_1}, \label{eq: No. 86' values R1}  \\
R_2 &= \frac{a_1}{b_1}, \label{eq: No. 86' values R2}  \\
R_3 &= \frac{\mathds{A} - a_1b_0}{b_0 b_1},
\label{eq: No. 86' values R3}  \\
C_1 &= \frac{b_1^2b_2}{a_1(2\mathds{A} - a_1b_0)C_2},
\label{eq: No. 86' values C1}
\end{align}
\end{subequations}
and $C_2$ is a positive root of
\begin{equation}   \label{eq: No. 86' values C2}
a_1\mathds{A}(2\mathds{A} - a_1b_0) C_2^2
- a_1 b_1^2 (2\mathds{A} - a_1b_0) C_2 + b_1^2b_2\mathds{A}
=0.
\end{equation}
}
\end{lemma}
\begin{proof}
\textit{Necessity.}
Since it is assumed that $R_1 = R_2$, \eqref{eq: No. 86 ka2}--\eqref{eq: No. 86 kb0} become
\begin{subequations}
\begin{align}
R_1^2 R_3C_1C_2 &= ka_2,    \label{eq: No. 86' ka2}  \\
(R_1 + R_3)R_1 C_1 + (R_1 + R_3)R_1C_2 &= ka_1,
\label{eq: No. 86' ka1}  \\
2R_1 + R_3 &= ka_0,  \label{eq: No. 86' ka0}  \\
(R_1 + 2R_3)R_1C_1C_2 &= kb_2,  \label{eq: No. 86' kb2}  \\
(R_1+R_3)C_1 + (R_1+R_3)C_2 &= kb_1,  \label{eq: No. 86' kb1}  \\
1 &= kb_0.  \label{eq: No. 86' kb0}
\end{align}
\end{subequations}
It is obvious that \eqref{eq: No. 86' kb0} is equivalent to
\begin{equation} \label{eq: No. 86' k}
k = \frac{1}{b_0}.
\end{equation}
From \eqref{eq: No. 86' ka1} and \eqref{eq: No. 86' kb1}, it follows that $R_1$ satisfies \eqref{eq: No. 86' values R1}, implying that $R_2$ satisfies \eqref{eq: No. 86' values R2}. Then, substituting  \eqref{eq: No. 86' values R1} and \eqref{eq: No. 86' k} into \eqref{eq: No. 86' ka0}, one concludes that $R_3$ satisfies
\eqref{eq: No. 86' values R3}, which implies $\mathds{A} - a_1b_0 > 0$.
Thus, it follows  from \eqref{eq: No. 86' kb2} and
\eqref{eq: No. 86' kb1} that $C_1$ satisfies
\eqref{eq: No. 86' values C1} and $C_2$ is a positive root of \eqref{eq: No. 86' values C2}.
Consequently, $2\mathds{A} - a_1b_0 > 0$.
Since the discriminant of  \eqref{eq: No. 86' values C2} should be nonnegative, one obtains condition~\eqref{eq: No. 86 network condition 01 R1 = R2}. Finally, substituting \eqref{eq: No. 86' values R1}--\eqref{eq: No. 86' values C2} and \eqref{eq: No. 86' k} into \eqref{eq: No. 86' ka2} yields $a_1b_2(\mathds{A} - a_1b_0) - a_2b_1 (2\mathds{A} - a_1b_0)  = 0$.
Since $\mathds{A} - a_1b_0 > 0$, condition~\eqref{eq: No. 86 network condition 02 R1 = R2} is obtained.

\textit{Sufficiency.}
Let the values of the elements satisfy \eqref{eq: No. 86' values R1}--\eqref{eq: No. 86' values C2}. Let $k$ satisfy
\eqref{eq: No. 86' k}. $\mathds{A} - a_1b_0 > 0$
and
condition~\eqref{eq: No. 86 network condition 01 R1 = R2} guarantee   all the elements to be positive and finite. Since condition~\eqref{eq: No. 86 network condition 02 R1 = R2} holds, it can be verified that \eqref{eq: No. 86 ka2}--\eqref{eq: No. 86 kb0} hold. Therefore, \eqref{eq: general impedance of No. 86 network} is equivalent to \eqref{eq: biquadratic impedance}.
\end{proof}


Following Lemmas~\ref{lemma: condition of the No. 86 network} and \ref{lemma: condition of the No. 86' network}, one can derive the following theorem, where the realizability condition of the configuration in Fig.~\ref{fig: Quartet-02-same-kind}(b) follows from that of Fig.~\ref{fig: Quartet-02-same-kind}(a) based on the principle of duality \cite{CWLC15}.

\begin{theorem}   \label{theorem: condition of the No. 86 network}
{A biquadratic impedance $Z(s) \in \mathcal{Z}_b$ is realizable as the configuration in Fig.~\ref{fig: Quartet-02-same-kind}(a) (resp. Fig.~\ref{fig: Quartet-02-same-kind}(b))  if and only if
$\mathds{R} > 0$ and $\mathds{R} - 4a_2b_0(a_1b_1-2\mathds{B}) \geq 0$ (resp. $\mathds{R} > 0$ and
$\mathds{R} - 4a_0b_2(a_1b_1 + 2\mathds{B}) \geq 0$).
}
\end{theorem}
\begin{proof}
First, one can show that the condition of Lemma~\ref{lemma: condition of the No. 86 network} is equivalent to
\begin{align}
&\mathds{R} > 0, \ \mathds{B} > 0 , \ \mathds{R} - 4a_2b_0(a_1b_1-2\mathds{B})  \geq 0,
\label{eq: No. 86 subcondition 02}  \\
&2a_2b_0\mathds{E}_b < b_2 (\mathds{R} + 2a_2b_0\mathds{B})  < 2a_0b_2\mathds{E}_b,
\label{eq: No. 86 subcondition 01}  \\
(\mathds{R} &+ 2a_2b_0\mathds{B})(\mathds{C} - 2a_2b_1)\mathds{C}+ 2(a_0a_2b_1^2-a_1^2b_0b_2)a_2\mathds{E}_b  > 0.
\label{eq: No. 86 subcondition 03}
\end{align}
Suppose that the condition of Lemma~\ref{lemma: condition of the No. 86 network} holds.
The discriminant of  \eqref{eq: equation for the root No. 86} is obtained as
$\delta = \mathds{R} (\mathds{R} - 4a_2b_0(a_1b_1-2\mathds{B}))$.
By Lemma~\ref{lemma: type of elements}, one has  $\mathds{R} > 0$.
Together with $\delta \geq 0$, one concludes that $\mathds{R} - 4a_2b_0(a_1b_1-2\mathds{B})  \geq 0$ must hold.  From
\eqref{eq: No. 86 positive restriction 01} and
\eqref{eq: No. 86 positive restriction 02}, one obtains
$\mathds{B} > 0$. Therefore, $\mathds{R} > 0$ indicates
$\mathds{A} \mathds{C} > \mathds{B}^2 > 0$, which further implies that $\mathds{R} + 2a_2b_0\mathds{B} > 0$ and
\begin{equation}
\mathds{E}_b = \frac{-b_2 \mathds{B}^2 + b_1 \mathds{B} \mathds{C}}{\mathds{B}}
> \frac{- b_2 \mathds{A} \mathds{C} +b_1\mathds{B} \mathds{C}}
{\mathds{B}}
= \frac{b_0 \mathds{C}^2}{\mathds{B}} > 0.
\end{equation}
From \cite[Ch.XV, Theorems 11 and 13]{Gan80},
$\mathds{R} > 0$ yields $\Delta_a > 0$ and $\Delta_b > 0$.
Substituting $R_3 = a_2/b_2$ and $R_3 = a_0/b_0$ into the left-hand side of  \eqref{eq: equation for the root No. 86}  yields, respectively,
\begin{align}
\left.b_0 \mathds{E}_b R_3^2 - (\mathds{R} + 2a_2b_0\mathds{B}) R_3 + a_2 \mathds{D}_a\right|_{R_3 = \frac{a_2}{b_2}} &=
\frac{a_2^2\Delta_b \mathds{B}}{b_2^2} > 0, \\
\left.b_0 \mathds{E}_b R_3^2 - (\mathds{R} + 2a_2b_0\mathds{B}) R_3 + a_2 \mathds{D}_a\right|_{R_3 = \frac{a_0}{b_0}} &= \Delta_a \mathds{B} > 0.
\end{align}
Thus, condition~\eqref{eq: No. 86 subcondition 01} holds. If
$\mathds{C} - 2a_2b_1 = 0$,  then condition~\eqref{eq: No. 86 subcondition 03} must hold because of \eqref{eq: No. 86 positive restriction 03}.
Otherwise, substituting $R_3 = - a_2(a_0a_2b_1^2 - a_1^2b_0b_2)/(b_0 \mathds{C} (\mathds{C} - 2a_2b_1))$ into the left-hand side of \eqref{eq: equation for the root No. 86}, one obtains
\begin{equation}
\frac{a_2^2 \mathds{E}_b (3a_1a_2b_0b_1-2a_0a_2b_1^2+a_0a_1b_1b_2-2a_1^2b_0b_2)^2}{b_0
\mathds{C}^2(\mathds{C} - 2a_2b_1)^2} \geq 0.
\end{equation}
Therefore, condition~\eqref{eq: No. 86 subcondition 03} is also satisfied.
Conversely, following the above discussion, one can also prove that  \eqref{eq: No. 86 subcondition 02}--\eqref{eq: No. 86 subcondition 03}    yield the condition of Lemma~\ref{lemma: condition of the No. 86 network}.

By \eqref{eq: from Z to Zc}, one converts   \eqref{eq: No. 86 subcondition 02}--\eqref{eq: No. 86 subcondition 03} into $W > 1$ as well as
\begin{equation}
\mathds{R}_c:= - 4U^2 - 4V^2 + 4UV(W+W^{-1})-(W-W^{-1})^2 > 0, \label{eq: condition 01}
\end{equation}
\begin{equation}
-4U^2 - 4V^2 + 4UV(W-3W^{-1}) - (W-W^{-1})(W-9W^{-1}) \geq 0, \label{eq: condition 02}
\end{equation}
\begin{equation} \label{eq: condition 03}
\begin{split}
- 4 U^2 - 4 V^2 &+ 4UV(W-W^{-1}) \\
&- (W-W^{-1})(W-5W^{-1}) +  8V^2 W^{-2}  > 0,
\end{split}
\end{equation}
\begin{equation}
- 4 V^2 + 4 U^2 + 4 U V (W-W^{-1}) - (W-W^{-1})(W+3W^{-1}) > 0, \label{eq: condition 04}
\end{equation}
and
\begin{equation}
4W^{-1}\zeta_{c}^{\ast}(UW-V)(UW-3V) + 8\lambda_c^{\ast}(V^2 - U^2) > 0.
\label{eq: condition 05}
\end{equation}
It is noted that condition~\eqref{eq: condition 01} yields $U>1$ and $V > 1$, and  condition~\eqref{eq: condition 02} yields $W \geq 3$. If $W = 3$, then $U = V$ by \eqref{eq: condition 02}, contradicting   condition~\eqref{eq: condition 05}. Hence, $W > 3$. Thus,
Lemma~\ref{lemma: U V W lemma} shows that
conditions~\eqref{eq: condition 01} and \eqref{eq: condition 02} with $W \neq 3$ imply condition~\eqref{eq: condition 05}.
If $UV \leq (W - W^{-1})/2$, then  conditions~\eqref{eq: condition 03} and \eqref{eq: condition 04} hold:
$- 4 U^2 - 4 V^2 + 4UV (W-W^{-1}) - (W- W^{-1})(W - 5W^{-1})+ 8V^2W^{-2} >  (W- W^{-1})^2 -  8UVW^{-1}  - (W- W^{-1})(W-5W^{-1}) +  8V^2W^{-2} = 4W^{-1}(W-W^{-1}) - 8UV W^{-1}  +  8V^2W^{-2}
\geq  8V^2W^{-2}  > 0$
and
$- 4 V^2 + 4 U^2 + 4 U V (W-W^{-1}) - (W-W^{-1})(W+ 3W^{-1}) > (W-W^{-1})^2 + 8U^2  -   8UVW^{-1}   - (W-W^{-1})(W+ 3W^{-1}) =
- 4W^{-1}(W - W^{-1}) - 8UV W^{-1}   + 8U^2 \geq -   8W^{-1} (W-W^{-1}) + 8U^2 >  8(U^2 - 1) > 0$,
because of condition~\eqref{eq: condition 01}.
Similarly, if $UV > (W - W^{-1})/2$, then one can also show that conditions~\eqref{eq: condition 03} and \eqref{eq: condition 04} hold
because of condition~\eqref{eq: condition 02}. Therefore, the condition of $W \neq 3$ and \eqref{eq: condition 01}--\eqref{eq: condition 02} together is equivalent to that of $W > 1$ and \eqref{eq: condition 01}--\eqref{eq: condition 05}. Moreover, through \eqref{eq: from Z to Zc}, the condition of Lemma~\ref{lemma: condition of the No. 86' network} is converted into
\begin{align}
2WV - 3U &> 0,   \label{eq: No. 86 condition 01 U V W} \\
2WU^2 + 2WV^2 - UV(W^2+3) &= 0,
\label{eq: No. 86 condition 02 U V W}  \\
U^2 + W^2V^2 + 3U^2V^2 - 2WUV^3 - 2WUV &\leq 0.
\label{eq: No. 86 condition 03 U V W}
\end{align}
When $W = 3$ and conditions~\eqref{eq: condition 01} and \eqref{eq: condition 02} hold, one obtains $U=V > 2\sqrt{3}/3$, implying that \eqref{eq: No. 86 condition 01 U V W}--\eqref{eq: No. 86 condition 03 U V W} hold.  The proof is completed if one can show that conditions~\eqref{eq: No. 86 condition 01 U V W}--\eqref{eq: No. 86 condition 03 U V W} can imply
conditions~\eqref{eq: condition 01} and \eqref{eq: condition 02}.
Indeed, by the following transformation
\begin{align}
U &= x \cos \frac{\pi}{4} - y \sin \frac{\pi}{4} = \frac{\sqrt{2}}{2} (x - y),  \\
V &= x \sin \frac{\pi}{4} + y \cos \frac{\pi}{4} =
\frac{\sqrt{2}}{2} (x + y),
\end{align}
conditions~\eqref{eq: No. 86 condition 01 U V W}--\eqref{eq: No. 86 condition 03 U V W} are further converted into $(2W-3)x + (2W+3)y > 0$, $(W+1)(W+3)y^2 - (W-1)(W-3)x^2 = 0$, and $(3-2W)x^4-4Wx^3y-6x^2y^2+4Wxy^3+(2W+3)y^4+2((W-1)x+(W+1)y)^2
\leq 0$ with $x > 0$,
which are in term equivalent to
\begin{align}
y = \pm \sqrt{\frac{(W-1)(W-3)}{(W+1)(W+3)}} x, \label{eq: No. 86 equivalent condition 01 U V W} \\
x \geq \frac{W+1}{2W}\sqrt{\frac{(W-1)(W+3)}{2}}.
\label{eq: No. 86 equivalent condition 02 U V W}
\end{align}
Next, conditions~\eqref{eq: condition 01} and \eqref{eq: condition 02} are converted into
\begin{equation}
\label{eq: equivalent condition 01}
\begin{split}
2W^{-1}(W-1)^2 x^2
- 2W^{-1}&(W+1)^2 y^2  \\
&- W^{-2}(W^2 - 1)^2   > 0,
\end{split}
\end{equation}
\begin{equation}
\label{eq: equivalent condition 02}
\begin{split}
2W^{-1}(W+&1)(W-3)x^2  \\
&- 2W^{-1}(W+3)(W-1)y^2 \\
&- W^{-2}(W^2-9)(W^2-1)  \geq 0.
\end{split}
\end{equation}
Substituting  \eqref{eq: No. 86 equivalent condition 01 U V W} into \eqref{eq: equivalent condition 01} and \eqref{eq: equivalent condition 02} gives
\begin{align}
\frac{8(W-1)x^2}{W+3} - \frac{(W+1)^2(W-1)^2}{W^2} &\geq 0,  \\
\frac{8(W-3)x^2}{W+1} - \frac{(W^2-1)(W^2-9)}{W^2} &\geq 0,
\end{align}
respectively.
It is obvious that that conditions~\eqref{eq: No. 86 equivalent condition 01 U V W} and \eqref{eq: No. 86 equivalent condition 02 U V W} can imply conditions~\eqref{eq: equivalent condition 01} and \eqref{eq: equivalent condition 02}.

Therefore, $Z_c(s) \in \mathcal{Z}_{b_c}$ is realizable as the configuration in Fig.~\ref{fig: Quartet-02-same-kind}(a)  if and only if conditions~\eqref{eq: condition 01} and \eqref{eq: condition 02} hold.
Through
\eqref{eq: from Zc to Z}, the condition for $Z(s) \in \mathcal{Z}_b$ is obtained as stated in the theorem.
\end{proof}

\begin{theorem}  \label{theorem: main theorem Res > 0}
{A biquadratic impedance $Z(s) \in \mathcal{Z}_b$ is realizable as a five-element bridge network containing two reactive elements of the same type if and only if $\mathds{R} > 0$ and  at least one of $(\mathds{R} - 4a_0b_2(a_1b_1+2\mathds{B}))$, $(\mathds{R} - 4a_2b_0(a_1b_1-2\mathds{B}))$, and $(\mathds{R} - 4a_0a_2b_0b_2)$ is nonnegative.
}
\end{theorem}
\begin{proof}
Combining Lemma~\ref{lemma: realization of Five-Element Bridge Networks with Two Reactive Elements of the Same Type} and Theorems~\ref{theorem: condition of the No. 85 network} and \ref{theorem: condition of the No. 86 network} yield the result.
\end{proof}

\subsection{Five-Element Bridge Networks with One Inductor and One Capacitor}
\label{subsec: different type}

\begin{lemma}  \label{lemma: all the possible quartets for one inductor and one capacitor}
{A biquadratic impedance $Z(s) \in \mathcal{Z}_b$ is realizable as a five-element bridge network with two reactive elements of different types if and only if $Z(s)$ is the impedance of one of configurations in
Figs.~\ref{fig: Quartet-01-different-kind}--\ref{fig: Quartet-03-different-kind}.}
\end{lemma}
\begin{proof}
This lemma is proved by a simple enumeration.
\end{proof}

\begin{figure}[thpb]
      \centering
      \includegraphics[scale=1.1]{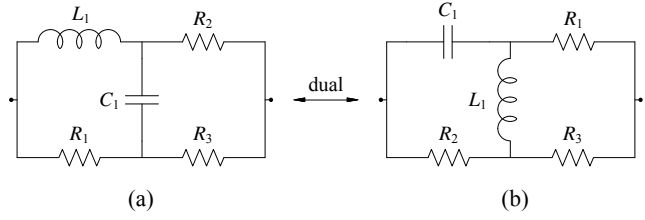}
      \caption{The two-reactive five-element bridge configurations containing different types of reactive elements, which are respectively supported by two one-terminal-pair labeled graphs $\mathcal{N}_3$ and $\text{Dual}(\mathcal{N}_3)$, where (a) is No.~70 configuration  in \cite{Lad48}, and (b) is No.~95 configuration in \cite{Lad48}.}
      \label{fig: Quartet-01-different-kind}
\end{figure}

\begin{figure}[thpb]
      \centering
      \includegraphics[scale=1.1]{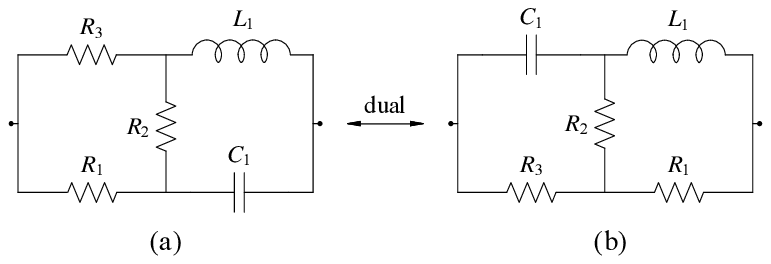}
      \caption{The two-reactive five-element bridge configurations containing different types of reactive elements, which are respectively supported by two one-terminal-pair labeled graphs $\mathcal{N}_4$ and $\text{Dual}(\mathcal{N}_4)$, where (a) is No.~105 configuration in \cite{Lad48}, and (b) is No.~107 configuration in \cite{Lad48}.}
      \label{fig: Quartet-02-different-kind}
\end{figure}

\begin{figure}[thpb]
      \centering
      \includegraphics[scale=1.1]{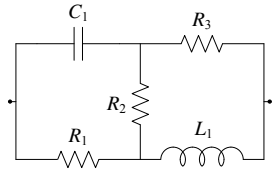}
      \caption{The two-reactive five-element bridge configuration containing different types of reactive elements, which is supported by a one-terminal-pair labeled graph $\mathcal{N}_5$ satisfying $\mathcal{N}_5 = \text{Dual}(\mathcal{N}_5)$, and it is No.~108 configuration in \cite{Lad48}.}
      \label{fig: Quartet-03-different-kind}
\end{figure}

The realizability condition of  Fig.~\ref{fig: Quartet-01-different-kind} has already been established in \cite{JS11}, as follows.

\begin{lemma}  \cite{JS11}
\label{lemma: condition of the quartet 01}
{A biquadratic impedance $Z(s) \in \mathcal{Z}_b$ is realizable as the configuration in Fig.~\ref{fig: Quartet-01-different-kind}(a)
(resp. Fig.~\ref{fig: Quartet-01-different-kind}(b)) if and only if
$\mathds{B} < 0$,
$\mathds{R} - 4a_0b_2(a_1b_1 + 2\mathds{B}) \leq 0$ (resp. $\mathds{B} > 0$, $\mathds{R} - 4a_2b_0(a_1b_1 - 2\mathds{B}) \leq 0$), and signs of $\mathds{D}_b$, $\mathds{E}_a$, and $(\mathds{R} - 2a_0b_2 \mathds{B})$
(resp. $\mathds{D}_a$, $\mathds{E}_b$, and $(\mathds{R} + 2a_2b_0 \mathds{B})$) are not all the same. If $\mathds{R} - 2a_0b_2\mathds{B} = 0$ (resp. $\mathds{R} + 2a_2b_0\mathds{B} = 0$), then $\mathds{D}_b \mathds{E}_a < 0$ (resp. $\mathds{D}_a \mathds{E}_b < 0$). }
\end{lemma}

By the star-mesh transformation \cite{Ver70}, it can be verified that the configuration  in Fig.~\ref{fig: Quartet-02-different-kind}(a) is equivalent to that in Fig.~\ref{fig: network of No. 104}. The element values for configurations in Figs.~\ref{fig: Quartet-02-different-kind}--\ref{fig: network of No. 104} have been listed in \cite{Lad48}, without any detail of derivation.

\begin{figure}[thpb]
      \centering
      \includegraphics[scale=1.2]{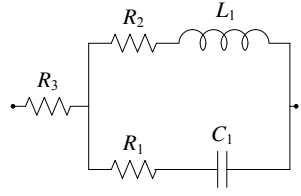}
      \caption{The two-reactive five-element series-parallel configuration that is equivalent to
      Fig.~\ref{fig: Quartet-02-different-kind}(a), which is No.~104 configuration in \cite{Lad48}.}
      \label{fig: network of No. 104}
\end{figure}

\begin{lemma}
\label{lemma: further condition of the No. 104 network}
{A biquadratic impedance $Z(s) \in \mathcal{Z}_b$ is realizable as the configuration in Fig.~\ref{fig: network of No. 104} if and only if
$\mathds{R} < 0$ and one of the following two conditions is satisfied:
\begin{enumerate}
  \item[1.] $\Gamma_a < 0$, either $\mathds{D}_b > 0$ for
  $\mathds{B} < 0$ or $\mathds{E}_b > 0$ for $\mathds{B} > 0$, and the signs of $\Delta_b$, $-\Delta_{ab}$, and $\Gamma_a$ are not all the same (when one of them is zero, the other two are nonzero and have different signs);
  \item[2.] $\Gamma_a > 0$, $\Delta_b > 0$, $\Delta_{ab} > 0$, and either $\mathds{D}_b + b_0 \mathds{B} < 0$ for $\mathds{B} < 0$ or $\mathds{E}_b - b_2 \mathds{B} < 0$ for $\mathds{B} > 0$.
\end{enumerate}
Furthermore, if the above condition is satisfied,
then the values of the elements are expressed as
\begin{subequations}
\begin{align}
R_1 &= \frac{a_2 - b_2R_3}{b_2},     \label{eq: No. 104 values R1}     \\
R_2 &= \frac{a_0 - b_0R_3}{b_0},     \label{eq: No. 104 values R2}     \\
L_1 &= \frac{\mathds{M}-2b_0b_2R_3}{b_0b_1},  \label{eq: No. 104 values L1}     \\
C_1 &= \frac{b_1b_2}{\mathds{M}-2b_0b_2R_3},  \label{eq: No. 104 values C1}
\end{align}
\end{subequations}
and $R_3$ is a positive root of
\begin{equation}  \label{eq: equation for the root No. 104}
b_0b_2 \Delta_b R_3^2 - 2b_0b_2\Delta_{ab} R_3 + \Gamma_a = 0,
\end{equation}
satisfying
\begin{equation}
R_3 < \min\{a_2/b_2, a_0/b_0\}.  \label{eq: No. 104 positive restriction}
\end{equation}}
\end{lemma}
\begin{proof}
\textit{Necessity.}
The impedance of the configuration in Fig.~\ref{fig: network of No. 104} is obtained as
\begin{equation}   \label{eq: general impedance of No. 104 network}
Z(s) = \frac{a(s)}{b(s)},
\end{equation}
where $a(s) = (R_1+R_3)L_1C_1s^2 + ((R_1R_3+R_2R_3+R_1R_2)C_1
+L_1)s + (R_2+R_3)$ and $b(s) = L_1C_1s^2 + (R_1+R_2)C_1s + 1$.
Thus,
\begin{subequations}
\begin{align}
(R_1+R_3)L_1C_1 &= ka_2, \label{eq: No. 104 ka2} \\
(R_1R_3+R_2R_3+R_1R_2)C_1 +L_1 &= ka_1, \label{eq: No. 104 ka1}  \\
R_2+R_3 &= ka_0,   \label{eq: No. 104 ka0} \\
L_1C_1 &= kb_2,   \label{eq: No. 104 kb2} \\
(R_1+R_2)C_1 &= kb_1,   \label{eq: No. 104 kb1}  \\
1 &= kb_0.   \label{eq: No. 104 kb0}
\end{align}
\end{subequations}
From \eqref{eq: No. 104 kb0}, one obtains
\begin{equation}  \label{eq: No. 104 k}
k = \frac{1}{b_0}.
\end{equation}
Based on \eqref{eq: No. 104 ka2} and \eqref{eq: No. 104 kb2}, one concludes  that $R_1$ satisfies \eqref{eq: No. 104 values R1}. From \eqref{eq: No. 104 ka0}, it follows that $R_2$ satisfies
\eqref{eq: No. 104 values R2}. Therefore,  condition~\eqref{eq: No. 104 positive restriction} must hold.
Substituting \eqref{eq: No. 104 values R1}, \eqref{eq: No. 104 values R2}, and \eqref{eq: No. 104 k} into \eqref{eq: No. 104 kb1} yields
the value of $C_1$ as \eqref{eq: No. 104 values C1}. As a result, the value of $L_1$ is obtained from
\eqref{eq: No. 104 kb2} as \eqref{eq: No. 104 values L1}. Finally, substituting \eqref{eq: No. 104 values R1}--\eqref{eq: No. 104 values C1} into \eqref{eq: No. 104 ka1} yields
\eqref{eq: equation for the root No. 104}.
It follows that the discriminant of  \eqref{eq: equation for the root No. 104} in $R_3$ is
\begin{equation}   \label{eq: discriminant of No. 104}
\delta  = (2b_0b_2\Delta_{ab})^2 - 4b_0b_2\Delta_b \Gamma_a
        = -4b_0b_1^2b_2\mathds{R}.
\end{equation}
Since  \eqref{eq: equation for the root No. 104} must have at least one positive root, one concludes that $\mathds{R} < 0$, and at most one of $\Delta_b$, $\Delta_{ab}$, and $\Gamma_a$ is zero.
Substituting $R_3 = a_0/b_0$, $R_3 = a_2/b_2$, and
$R_3 = (a_0/b_0 + a_2/b_2)/2 = \mathds{M}/(2b_0b_2)$  into the left-hand side of \eqref{eq: equation for the root No. 104}, one obtains, respectively,
\begin{align}
\left.b_0b_2 \Delta_b R_3^2 - 2b_0b_2\Delta_{ab} R_3 + \Gamma_a
\right|_{R_3 = \frac{a_0}{b_0}}
&= -\frac{\mathds{B}\mathds{D}_b}{b_0},
\label{eq: subs R3 01 No. 104}
\\
\left.b_0b_2 \Delta_b R_3^2 - 2b_0b_2\Delta_{ab} R_3 + \Gamma_a
\right|_{R_3 = \frac{a_2}{b_2}}
&= \frac{\mathds{B}\mathds{E}_b}{b_2},
\label{eq: subs R3 02 No. 104}
\\
\left.b_0b_2 \Delta_b R_3^2 - 2b_0b_2\Delta_{ab} R_3 + \Gamma_a
\right|_{R_3 = \frac{\mathds{M}}{2b_0b_2}}
&= \frac{b_1^2 \mathds{B}^2}{4b_0b_2} > 0.
\label{eq: subs R3 03 No. 104}
\end{align}

Since the condition of Lemma~\ref{lemma: condition of at most four} does not hold, $\Gamma_a \neq 0$.
When $\Gamma_a < 0$, based on \eqref{eq: subs R3 03 No. 104} it follows that  \eqref{eq: equation for the root No. 104} has only one positive root in $R_3$ such that \eqref{eq: No. 104 positive restriction} holds. Therefore, the signs of $\Delta_b$, $-\Delta_{ab}$, and $\Gamma_a$ are not all the same (when one of them is zero, the other two are nonzero and have different signs). Moreover, if
$\mathds{B} < 0$, then $a_0/b_0 < \mathds{M}/(2b_0b_2) < a_2/b_2$, implying that
$\mathds{D}_b > 0$ to guarantee \eqref{eq: subs R3 01 No. 104} to be positive; if
$\mathds{B} > 0$, then $a_0/b_0 > \mathds{M}/(2b_0b_2) >  a_2/b_2$, implying that
$\mathds{E}_b > 0$ to guarantee \eqref{eq: subs R3 02 No. 104} to be positive.

When $\Gamma_a > 0$, based on \eqref{eq: subs R3 03 No. 104} it follows that $\Delta_b > 0$ and $-\Delta_{ab}<0$. If
$\mathds{B} < 0$, then
$a_0/b_0 < \mathds{M}/(2b_0b_2) < a_2/b_2$. Therefore, in either the case when \eqref{eq: subs R3 01 No. 104} is negative or the case when \eqref{eq: subs R3 01 No. 104} is nonnegative, one has  $\Delta_{ab}/\Delta_b < a_0/b_0$ holds. The above two cases correspond to
$-b_1 \mathds{A} < -b_0 \mathds{B}$ and $-b_0 \mathds{B} \leq - b_1 \mathds{A} < - 2 b_0 \mathds{B}$, respectively.  Hence, combining them yields
$-b_1 \mathds{A} < -2 b_0 \mathds{B}$,  which is equivalent to
$\mathds{D}_b + b_0\mathds{B} < 0$. Similarly, if
$\mathds{B} > 0$, then  $\mathds{E}_b - b_2\mathds{B} < 0$.

\textit{Sufficiency.}
Let the values of the elements in Fig.~\ref{fig: network of No. 104} satisfy \eqref{eq: No. 104 values R1}--\eqref{eq: No. 104 values C1}, and $R_3$ be a positive root of \eqref{eq: equation for the root No. 104} satisfying \eqref{eq: No. 104 positive restriction}. Then, $a_2 - b_2R_3 > 0$, $a_0 - b_0R_3 > 0$, and $\mathds{M} - 2b_0b_2R_3 > 0$. Letting $k$ satisfy \eqref{eq: No. 104 k}, it can be verified that \eqref{eq: No. 104 ka2}--\eqref{eq: No. 104 kb0} hold.
$\mathds{R} < 0$ implies that the discriminant of   \eqref{eq: equation for the root No. 104} as expressed in \eqref{eq: discriminant of No. 104} is positive.

If condition~1  is satisfied, then
as discussed in the necessity part  there exists a unique positive root of  \eqref{eq: equation for the root No. 104} in terms of $R_3$ such that \eqref{eq: No. 104 positive restriction} holds.

If condition~2 holds, and either
$-b_1 \mathds{A} < -b_0 \mathds{B}$ for $\mathds{B} < 0$ or
$-b_2 \mathds{B} < -b_1 \mathds{C}$ for $\mathds{B} > 0$, then it can be proved that there exists a unique positive root for  \eqref{eq: equation for the root No. 104} in terms of $R_3$ such that \eqref{eq: No. 104 positive restriction} holds.

If condition~2 holds, and either
$-b_0 \mathds{B} \leq -b_1 \mathds{A} < -2b_0 \mathds{B}$ for $\mathds{B} < 0$
or $-2b_2 \mathds{B} < -b_1 \mathds{C} \leq -b_2 \mathds{B}$
for $\mathds{B} > 0$, then there are two positive roots for \eqref{eq: equation for the root No. 104} in terms of $R_3$ such that \eqref{eq: No. 104 positive restriction} holds.

As a conclusion, the values of elements must be positive and finite. The given impedance $Z(s)$ is realizable as the specified network.
\end{proof}

The values of elements in Fig.~\ref{fig: Quartet-02-different-kind}(a) can be obtained from those in Fig.~\ref{fig: network of No. 104} via the following transformation: $R_P/R_2 \rightarrow R_1$, $R_P/R_3 \rightarrow R_2$, $R_P/R_1 \rightarrow R_3$, $C_1 \rightarrow C_1$, and $L_1 \rightarrow L_1$, where $R_P = R_1R_2 + R_2R_3 + R_3R_1$.

Since the realizability condition of the configuration in Fig.~\ref{fig: Quartet-02-different-kind}(a) is equivalent to that of Lemma~\ref{lemma: further condition of the No. 104 network}, a necessary and sufficient condition for the realizability of the configurations in Fig.~\ref{fig: Quartet-02-different-kind} is obtained as follows.

\begin{theorem}  \label{theorem: condition of the quartet 02}
{A biquadratic impedance $Z(s) \in \mathcal{Z}_b$ is realizable as one of  the configurations in Fig.~\ref{fig: Quartet-02-different-kind} if and only if
$\mathds{R} < 0$ and one of the following three conditions
is satisfied:
\begin{enumerate}
  \item[1.] $\Gamma_a < 0$, either $\mathds{D}_b > 0$ for $\mathds{B} < 0$ or $\mathds{E}_b > 0$ for $\mathds{B} > 0$, and the signs of $\Delta_b$, $-\Delta_{ab}$, and $\Gamma_a$ are not all the same (when one of them is zero, the other two are nonzero and have different signs);
  \item[2.] $\Gamma_b < 0$, either $\mathds{D}_a > 0$ for $\mathds{B} > 0$ or $\mathds{E}_a > 0$ for $\mathds{B} < 0$, and the signs of $\Delta_a$, $-\Delta_{ab}$, and $\Gamma_b$ are not all the same (when one of them is zero, the other two are nonzero and have different signs);
  \item[3.] $\Gamma_a > 0$, $\Gamma_b > 0$, and $\Delta_{ab} > 0$.
\end{enumerate}}
\end{theorem}
\begin{proof}
Conditions~1 and 2 can be obtained from
Lemma~\ref{lemma: further condition of the No. 104 network} based on the principle of duality \cite{CWLC15}.
To obtain condition~3, it suffices to show that
\begin{equation}  \label{eq: third condition canonical}
\Gamma_{a_c} > 0, \ \Gamma_{b_c} > 0, \ \Delta_{ab_c} > 0
\end{equation}
is equivalent to the union of the following two conditions:
\begin{enumerate}
  \item[$a$.] $\Gamma_{a_c} > 0$, $V > 1$, $\Delta_{ab_c} > 0$, and either $2UV-2WV^2+(W-W^{-1})<0$ for $W<1$ or
      $2UV-2W^{-1}V^2-(W-W^{-1}) < 0$ for
      $W>1$;
  \item[$b$.] $\Gamma_{b_c} > 0$, $U > 1$, $\Delta_{ab_c} > 0$, and either $2UV-2WU^2+(W-W^{-1}) < 0$ for $W<1$ or $2UV-2W^{-1}U^2-(W-W^{-1}) < 0$ for $W>1$.
\end{enumerate}
First, one can show that condition~$a$ or $b$ implies \eqref{eq: third condition canonical}.
Without loss of generality,  assume that  $W > 1$. Then, one obtains $W+W^{-1} < 2UV < (W-W^{-1})+2V^2W^{-1}$ from condition~$a$.
Hence, it follows from condition~$a$ that
$\Gamma_{b_c} = - 4 U^2 + 4UV(W+W^{-1}) - (W+W^{-1})^2 = - V^{-2}(2UV - V(V-\sqrt{V^2 -1})(W+W^{-1}))(2UV - V(V+\sqrt{V^2 -1})(W+W^{-1})) > 0$,
since
$(W+W^{-1})  - V(V-\sqrt{V^2-1})(W+W^{-1}) = \sqrt{V^2 - 1}(V-\sqrt{V^2 - 1})(W+W^{-1}) > 0$ and
$(W-W^{-1}) + 2V^2 W^{-1} - V(V+\sqrt{V^2-1})(W+W^{-1}) =
- (W-W^{-1})(V^2-1)-V\sqrt{V^2-1}(W+W^{-1}) < 0$.
Therefore, condition~$a$  yields condition~\eqref{eq: third condition canonical}.
Similarly, condition~$b$ implies $\Gamma_{a_c} > 0$, which also yields condition~\eqref{eq: third condition canonical}. In addition, the case of $W < 1$ can be similarly  proved.

Now, it remains to show that condition~\eqref{eq: third condition canonical} implies condition~$a$ or $b$. Assume that $W > 1$. Since $\Gamma_{a_c} > 0$ and $\Gamma_{b_c} > 0$ can  yield respectively $U > 1$ and $V > 1$, if $2UV - 2W^{-1}V^2 - (W-W^{-1}) < 0$ then condition~$a$ holds. Otherwise, one obtains
$U - \sqrt{U^2 - 2W^{-1}(W-W^{-1})} \leq 2V W^{-1} \leq
U + \sqrt{U^2 - 2W^{-1}(W-W^{-1})}$. It can be verified that
$U(W + W^{-1}) + (W - W^{-1})\sqrt{U^2 - 1} - W (U + \sqrt{U^2 - 2W^{-1}(W-W^{-1})}) > 0$. Together with $\mathds{R}_c < 0$, one has
$V < (U(W+W^{-1})-(W-W^{-1})\sqrt{U^2-1})/2$, which implies that
$2 U V - 2 W^{-1} U^2 - (W - W^{-1}) < U^2(W+W^{-1}) - (W - W^{-1}) U \sqrt{U^2 - 1} - 2W^{-1} U^2 - (W - W^{-1}) = (W-W^{-1})\sqrt{U^2- 1}(\sqrt{U^2 - 1}-U) < 0$. Hence, condition~$b$ is obtained. The case of $W < 1$ can be similarly proved.
\end{proof}

\begin{theorem}  \label{theorem: condition of the quartet 03}
{A biquadratic impedance $Z(s) \in \mathcal{Z}_b$ is realizable as the configuration in Fig.~\ref{fig: Quartet-03-different-kind} if and only if $\mathds{R} < 0$ and the signs of $\Gamma_a$, $\Gamma_b$, and $(\mathds{M}\mathds{R} + 2a_0a_2b_0b_2\Delta_{ab})$ are not all the same (when $\mathds{M}\mathds{R} + 2a_0a_2b_0b_2\Delta_{ab} = 0$, $\Gamma_a \Gamma_b < 0$). Furthermore, if the above condition holds, then the values of the elements are expressed as
\begin{subequations}
\begin{align}
R_1 &= \frac{a_0}{b_0},
\label{eq: No. 108 R1}    \\
R_3 &= \frac{a_2}{b_2},
\label{eq: No. 108 R3}    \\
L_1 &=
\frac{(a_1a_2b_0 + a_0 \mathds{C})R_2 + a_0a_1a_2}
{(b_0R_2+a_0)\mathds{M}},
\label{eq: No. 108 L1}    \\
C_1 &= \frac{b_0b_1b_2R_2 + (a_0b_1b_2 - b_0\mathds{C})}
{(b_0R_2+a_0)\mathds{M}},
\label{eq: No. 108 C1}
\end{align}
\end{subequations}
and $R_2$ is a positive root of
\begin{equation}   \label{eq: equation for the root QuadEq}
b_0b_2\Gamma_a R_2^2 + (\mathds{M}\mathds{R} + 2a_0a_2b_0b_2\Delta_{ab}) R_2 + a_0a_2\Gamma_b = 0,
\end{equation}
satisfying
\begin{subequations}
\begin{align}
b_0b_1b_2R_2 + (a_0b_1b_2 - b_0\mathds{C})   &> 0,
\label{eq: No. 108 positive condition 01}  \\
(a_1a_2b_0 + a_0 \mathds{C})R_2 + a_0a_1a_2  &> 0.
\label{eq: No. 108 positive condition 02}
\end{align}
\end{subequations}
}
\end{theorem}
\begin{proof}
\textit{Necessity.}
The impedance of the configuration in Fig.~\ref{fig: Quartet-03-different-kind} is given by
\begin{equation}   \label{eq: general impedance of No. 108 network}
Z(s) = \frac{a(s)}{b(s)},
\end{equation}
where $a(s) = (R_1+R_2)R_3L_1C_1s^2 + (R_1R_2R_3C_1 +
(R_1 + R_2 + R_3)L_1)s + (R_2+R_3)R_1$ and $b(s) = (R_1+R_2)L_1C_1s^2 + ((R_1R_2 + R_2R_3 + R_1R_3)C_1 + L_1)s+R_2+R_3$.
Thus,
\begin{subequations}
\begin{align}
(R_1+R_2)R_3L_1C_1 &= ka_2,  \label{eq: No. 108 ka2}  \\
R_1R_2R_3C_1 + (R_1 + R_2 + R_3)L_1 &= ka_1,  \label{eq: No. 108 ka1}  \\
(R_2+R_3)R_1 &= ka_0,  \label{eq: No. 108 ka0}  \\
(R_1+R_2)L_1C_1 &= kb_2, \label{eq: No. 108 kd2}   \\
(R_1R_2 + R_2R_3 + R_1R_3)C_1 + L_1 &= kb_1,  \label{eq: No. 108 kd1}  \\
R_2+R_3 &= kb_0.  \label{eq: No. 108 kd0}
\end{align}
\end{subequations}
From \eqref{eq: No. 108 ka2} and \eqref{eq: No. 108 kd2}, it follows that $R_3$ satisfies \eqref{eq: No. 108 R3}.
From \eqref{eq: No. 108 ka0} and \eqref{eq: No. 108 kd0}, it follows that $R_1$ satisfies \eqref{eq: No. 108 R1}.
Substituting \eqref{eq: No. 108 R3} into \eqref{eq: No. 108 kd0} yields
\begin{equation}  \label{eq: k No. 108}
k = \frac{b_2R_2 + a_2}{b_0b_2}.
\end{equation}
Therefore, $L_1$ and $C_1$  can be solved from
\eqref{eq: No. 108 ka1} and \eqref{eq: No. 108 kd1} as
\eqref{eq: No. 108 L1} and \eqref{eq: No. 108 C1}. The assumption that all the values of the elements are positive and finite implies conditions~\eqref{eq: No. 108 positive condition 01} and \eqref{eq: No. 108 positive condition 02}.
Substituting \eqref{eq: No. 108 R1}--\eqref{eq: No. 108 C1} into
\eqref{eq: No. 108 kd2} yields  \eqref{eq: equation for the root QuadEq}.
The discriminant of  \eqref{eq: equation for the root QuadEq} in terms of $R_2$ is obtained as
\begin{equation}    \label{eq: delta positive}
\begin{split}
\delta
=  \mathds{M}^2 \mathds{R} (\mathds{R} - 4a_0a_2b_0b_2).
\end{split}
\end{equation}
Since the discriminant must be nonnegative to guarantee the existence of real roots, together with Lemma~\ref{lemma: type of elements} one has $\mathds{R} < 0$, implying that $\delta > 0$ and at most one of $\Gamma_a$, $\Gamma_b$, and $(\mathds{M}\mathds{R} + 2a_0a_2b_0b_2\Delta_{ab})$ is zero.
If one of them is zero, then it is only possible that $\mathds{M}\mathds{R} + 2a_0a_2b_0b_2\Delta_{ab} = 0$ and $\Gamma_a \Gamma_b < 0$. If none of them is zero, then it follows that the signs of them cannot be the same to guarantee the existence of the positive root.

\textit{Sufficiency.}
Let the values of the elements in
Fig.~\ref{fig: Quartet-03-different-kind} be
\eqref{eq: No. 108 R1}--\eqref{eq: No. 108 C1}, and $R_2$ be a positive root of  \eqref{eq: equation for the root QuadEq} satisfying conditions~\eqref{eq: No. 108 positive condition 01} and \eqref{eq: No. 108 positive condition 02}. Let $k$ satisfy \eqref{eq: k No. 108}. It can be verified that \eqref{eq: No. 108 ka2}--\eqref{eq: No. 108 kd0} hold, implying that \eqref{eq: general impedance of No. 108 network} is equivalent to \eqref{eq: biquadratic impedance}.

It suffices to show that \eqref{eq: equation for the root QuadEq} always has a positive root, such that $R_1$, $R_3$, $L_1$, $C_1$ expressed as
\eqref{eq: No. 108 R1}--\eqref{eq: No. 108 C1} are positive. Since $R_1$ and $R_3$ expressed as \eqref{eq: No. 108 R1} and
\eqref{eq: No. 108 R3} are obviously positive, one only needs to discuss $L_1$ and  $C_1$.

It is not difficult to see that if the signs of $\Gamma_a$, $\Gamma_b$, and $(\mathds{M}\mathds{R} + 2a_0a_2b_0b_2\Delta_{ab})$ satisfy the given conditions, then  \eqref{eq: equation for the root QuadEq} must have at least one positive root, since $\mathds{R} < 0$ implies  that the discriminant of  \eqref{eq: equation for the root QuadEq} shown in \eqref{eq: delta positive} is always positive.
Furthermore,
$- k^4 \mathds{R} = (R_1+R_2)^2(R_2+R_3)^2(R_1R_3C_1-L_1)^2L_1C_1$.
Together with $\mathds{R} < 0$, it follows that    $L_1$ and $C_1$ are both positive or negative. Hence, $\chi_1\chi_2 > 0$, where $\chi_1 := b_0b_1b_2R_2 + (a_0b_1b_2 - b_0 \mathds{C})$ and $\chi_2  := (a_1a_2b_0 + a_0 \mathds{C})R_2 + a_0a_1a_2$.
Assume that $\chi_1 < 0$ and $\chi_2 < 0$. Then, by letting $a_0a_2 \chi_1 + b_0b_2 \chi_2$, one obtains
$a_0a_2 \chi_1 + b_0b_2 \chi_2 = (a_1b_0b_2 R_2 + a_0a_2b_1) \mathds{M} < 0$.
This contradicts   the assumption that all the coefficients are positive. Hence, $\chi_1 > 0$ and $\chi_2 > 0$, suggesting that $L_1 > 0$  and $C_1 > 0$.
\end{proof}

Combining Lemma~\ref{lemma: condition of the quartet 01}, Theorems~\ref{theorem: condition of the quartet 02} and \ref{theorem: condition of the quartet 03}, one obtains the following result.

\begin{theorem}  \label{theorem: main theorem Res < 0}
{A biquadratic impedance $Z(s) \in \mathcal{Z}_b$ is realizable as a five-element bridge network containing one inductor and one capacitor if and only if $\mathds{R} < 0$, and $Z(s)$ is regular\footnote{
A necessary and sufficient condition for a biquadratic impedance to be regular is presented in \cite[Lemma~5]{JS11}.} or satisfies the condition of Lemma~\ref{lemma: condition of the quartet 01}.}
\end{theorem}
\begin{proof}
\textit{Necessity.}
By Lemma~\ref{lemma: type of elements}, $\mathds{R} < 0$. It is shown in \cite{JS11} that biquadratic impedances that can realize
configurations in Figs.~\ref{fig: Quartet-02-different-kind} and
\ref{fig: Quartet-03-different-kind} must be regular.
Based on Lemma~\ref{lemma: all the possible quartets for one inductor and one capacitor}, the necessity part is proved.

\textit{Sufficiency.}
Based on Lemma~\ref{lemma: condition of the quartet 01}, one only needs to consider the case when $\mathds{R} < 0$ and $Z(s)$ is regular, which means that the corresponding $Z_c(s)$ is regular.
Assuming that the condition of Lemma~\ref{lemma: condition of at most four} does not hold, $Z_c(s)$ is regular if and only if (1) $\lambda_c > 0$ or $\lambda_c^{\dag} > 0$ when $W < 1$;
(2) $\lambda_c^{\ast} > 0$ or $\lambda_c^{\ast\dag} > 0$ when $W > 1$. It suffices to show that if $\mathds{R}_c < 0$ and $Z_c(s) \in \mathcal{Z}_{b_c}$ is regular then $Z_c(s)$ is realizable as one of the configurations in Figs.~\ref{fig: Quartet-02-different-kind} and \ref{fig: Quartet-03-different-kind}.

Case~1: $\Gamma_{a_c} < 0$ and $\Gamma_{b_c} < 0$. If $U < 1$ and $V < 1$, then $\Delta_{ab_c} = 4UV - 2(W+W^{-1}) < 0$. Suppose that $U \geq 1$.
Then, $\Gamma_{a_c} < 0$ implies that $V < (W+W^{-1})(U-\sqrt{U^2 - 1})/2$ or
$V > (W+W^{-1})(U+\sqrt{U^2 - 1})/2$, and $\Gamma_{b_c} < 0$ implies
$V < (4U^2 + (W+W^{-1})^2)/(4U(W+W^{-1}))$. Since
$2U(W+W^{-1})^2(U+\sqrt{U^2-1}) - ((W+W^{-1})^2 + 4U^2) > 4(2U(U+\sqrt{U^2-1})-1) - 4U^2 = 4(U^2 - 1) + 8U\sqrt{U^2 - 1}  > 0$, it is only possible that $V <  (W+W^{-1})(U-\sqrt{U^2 - 1})/2$.
Hence,
$\Delta_{ab_c} = 4 U V - 2 (W + W^{-1}) < 2 (W + W^{-1})(U^2 - U\sqrt{U^2 - 1}) - 2 (W + W^{-1}) =2 (W + W^{-1}) \sqrt{U^2 - 1}(\sqrt{U^2 - 1} - U) < 0$.
Making use of Theorem~\ref{theorem: condition of the quartet 02} and \eqref{eq: from Z to Zc}, $Z_c(s)$ is realizable as one of the configurations in Fig.~\ref{fig: Quartet-02-different-kind}.

Case~2: Only one of $\Gamma_{a_c}$ and $\Gamma_{b_c}$ is negative. By Theorem~\ref{theorem: condition of the quartet 03} and \eqref{eq: from Z to Zc}, $Z_c(s)$ is realizable as the configuration  in Fig.~\ref{fig: Quartet-03-different-kind}.

Case~3: $\Gamma_{a_c} > 0$ and $\Gamma_{b_c} > 0$. One obtains that $U > 1$ and $V > 1$. If $\Delta_{ab_c} > 0$,
then $Z_c(s)$ is realizable as one of the configurations in Fig.~\ref{fig: Quartet-02-different-kind} by Theorem~\ref{theorem: condition of the quartet 02} and \eqref{eq: from Z to Zc}. If $\Delta_{ab_c} \leq 0$, then
$\mathds{R}_c = - 4 U^2 - 4 V^2 + 4 U V (W+W^{-1}) - (W-W^{-1})^2 < 0$
yields
$- (W+W^{-1})^3 + 4UV(W+W^{-1})^2 - 4(U^2+V^2)(W+W^{-1}) + 8UV  < - (W+W^{-1})^3 + 4UV(W+W^{-1})^2 + ((W-W^{-1})^2-4UV(W+W^{-1}))(W+W^{-1}) + 8UV = -4(W+W^{-1}) + 8UV = 2 \Delta_{ab_c}  \leq 0$. Therefore, $Z_c(s)$ is realizable as the configuration in Fig.~\ref{fig: Quartet-03-different-kind}  based on Theorem~\ref{theorem: condition of the quartet 03} and \eqref{eq: from Z to Zc}.
\end{proof}

\begin{corollary}  \label{corollary: 02}
{A biquadratic impedance $Z(s) \in \mathcal{Z}_b$ with $\mathds{R} < 0$ is regular if and only if it is realizable as one of the configurations in Figs.~\ref{fig: Quartet-02-different-kind} and \ref{fig: Quartet-03-different-kind}.}
\end{corollary}
\begin{proof}
This corollary is obtained based on the proof of Theorem~\ref{theorem: main theorem Res < 0}.
\end{proof}

\begin{corollary}  \label{corollary: 03}
{A biquadratic impedance, which can be realized as an irreducible\footnote{An irreducible network means that it can never become equivalent to the one containing fewer elements.} five-element series-parallel network containing one inductor and one capacitor,  can always be realized as a five-element bridge network containing
one inductor and one capacitor.}
\end{corollary}
\begin{proof}
It has been proved in \cite{JS11} that if the biquadratic impedance $Z(s)$ in the form of \eqref{eq: biquadratic impedance} is realizable as a five-element series-parallel network containing  one inductor and one capacitor, then $Z(s)$ must be regular. Since the network is irreducible, it follows that $Z(s) \in \mathcal{Z}_b$ and $\mathds{R} < 0$. Hence, the conclusion directly follows  from Theorem~\ref{theorem: main theorem Res < 0}.
\end{proof}

\begin{corollary}
{If a biquadratic impedance $Z(s) \in \mathcal{Z}_b$ can be realized as a network containing one inductor, one capacitor, and at least three resistors, then the network will always be equivalent to a five-element bridge network containing one inductor and one capacitor.}
\end{corollary}
\begin{proof}
It can be proved by Theorem~\ref{theorem: main theorem Res < 0} and a theorem of Reichert \cite{Rei69}.
\end{proof}

\subsection{Summary and Notes}  \label{subsec: summary and notes}

\begin{theorem}  \label{theorem: final condition}
{A biquadratic impedance $Z(s) \in \mathcal{Z}_b$ is realizable as   a two-reactive five-element bridge network if and only if one of the following two conditions holds:
\begin{enumerate}
  \item[1.] $\mathds{R} > 0$,  and at least one of $(\mathds{R} - 4a_0b_2(a_1b_1+2\mathds{B}))$, $(\mathds{R} - 4a_2b_0(a_1b_1 - 2\mathds{B}))$ and $(\mathds{R} - 4a_0a_2b_0b_2)$ is nonnegative;
  \item[2.] $\mathds{R} < 0$, and $Z(s)$ is regular or satisfies the condition of Lemma~\ref{lemma: condition of the quartet 01}.
\end{enumerate}
}
\end{theorem}
\begin{proof}
Combining Theorems~\ref{theorem: main theorem Res > 0} and \ref{theorem: main theorem Res < 0} leads to the conclusion.
\end{proof}

Now, a corresponding result for general five-element bridge networks directly follows.

\begin{theorem}  \label{theorem: final corollary}
{A biquadratic impedance $Z(s) \in \mathcal{Z}_b$ is realizable as  a five-element bridge network if and only if $Z(s)$ satisfies the condition of Theorem~\ref{theorem: final condition} or is realizable as a configuration in Fig.~\ref{fig: Three-Reactive-Quartet}.\footnote{A necessary and sufficient condition for the realizability of Fig.~\ref{fig: Three-Reactive-Quartet} is presented in \cite[Theorem 7]{JS11}}
}
\end{theorem}
\begin{figure}[thpb]
      \centering
      \includegraphics[scale=1.1]{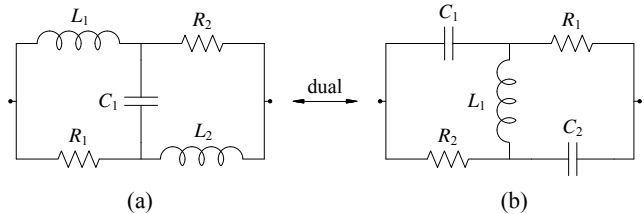}
      \caption{The three-reactive five-element bridge configurations in Theorem~\ref{theorem: final corollary}, which are respectively supported by two one-terminal-pair labeled graphs $\mathcal{N}_6$ and $\text{Dual}(\mathcal{N}_6)$. }
      \label{fig: Three-Reactive-Quartet}
\end{figure}
\begin{proof}
\textit{Sufficiency.} The sufficiency part is obvious.

\textit{Necessity.} Since $\mathds{R} \neq 0$, the \textit{McMillan degree} (see \cite[Chapter~3.6]{AV73}) of $Z(s) \in \mathcal{Z}_b$ satisfies $\delta(Z(s)) = 2$. Since the McMillan degree is equal to the minimal number of reactive elements for realizations of $Z(s)$ \cite[pg.~370]{AV73}, there must exist at least two reactive elements. Since $Z(s) \in \mathcal{Z}_b$ has no pole or zero on $j \mathbb{R} \cup \infty$, the number of reactive elements cannot be five. If the number of reactive elements is four (only one resistor), then $Z(0) = Z(\infty)$, which is equal to the value of the resistor. This means that $\mathds{B} = 0$, contradicting  the assumption that the condition of Lemma~\ref{lemma: condition of at most four} does not hold. Therefore, the number of reactive elements must be two or three.
If the number of reactive elements is two, then the condition of Theorem~\ref{theorem: final condition} holds.
If the number of reactive elements is three, then their types cannot be the same by Lemma~\ref{lemma: topological structure}. Furthermore, by the discussion in \cite{Lad64}, the network is equivalent to either a five-element series-parallel structure containing one inductor and one capacitor  or a configuration in Fig.~\ref{fig: Three-Reactive-Quartet}.
By Corollary~\ref{corollary: 03}, the theorem is proved.
\end{proof}

The condition of Theorem~\ref{theorem: final condition} can be converted into a condition in terms of the canonical form $Z_c(s) \in \mathcal{Z}_{b_c}$ through \eqref{eq: from Z to Zc}, which is  shown on the $U$--$V$ plane in Fig.~\ref{fig: UV} when $W = 2$.
If $(U,V)$ is within the shaded region (excluding  the inside curves $\Gamma_{a_c} \Gamma_{b_c} = 0$ and $\lambda_c^{\ast} \lambda_c^{\ast\dag} = 0$), then $Z_c(s)$ is realizable as   two-reactive five-element bridge networks. The hatched region ($\sigma_c < 0$) represents the non-positive-realness case, where $Z_c(s)$ cannot be realized as a passive network.

\begin{figure}[thpb]
      \centering
      \includegraphics[scale=0.55]{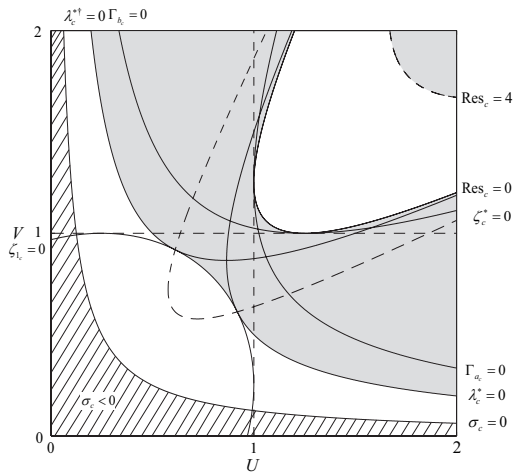}
      \caption{The $U$-$V$ plane showing a necessary and sufficient condition for any $Z_c(s) \in \mathcal{Z}_{b_c}$ to be realizable as the two-reactive five-element bridge network, when $W = 2$.}
      \label{fig: UV}
\end{figure}

Some important notes  are listed as follows.

\begin{remark}
{Ladenheim \cite{Lad48} only listed element values for configurations in Figs.~\ref{fig: Quartet-01-same-kind}--\ref{fig: Quartet-03-different-kind} without showing any detail of derivation. Neither explicit conditions   in terms of the coefficients of $Z(s)$  only (like the condition of Theorem~\ref{theorem: condition of the No. 85 network}) nor a complete set of conditions are  given in \cite{Lad48}. Besides, the special case when $R_1 = R_2$ for Fig.~\ref{fig: Quartet-02-same-kind}(a) (No.~86 configuration in \cite{Lad48}) is not discussed in \cite{Lad48}.}
\end{remark}

\begin{remark}
{In \cite{JS11},  necessary and sufficient  conditions are derived only for the realizability of bridge networks that sometimes cannot   realize regular biquadratic impedances. In the present paper, through discussing other two-reactive five-element bridge networks and by combining their conditions, a complete result is obtained (Theorem~\ref{theorem: final condition}). Together with previous results in \cite{Lad64}, the result has been further  extended to the general five-element bridge case (Theorem~\ref{theorem: final corollary}). }
\end{remark}

\begin{remark}
{Corollary~\ref{corollary: 02} shows that the regular biquadratic impedance $Z(s) \in \mathcal{Z}_b$ with $\mathds{R} < 0$ is always realizable as one of the three five-element bridge configurations in Figs.~\ref{fig: Quartet-02-different-kind} and \ref{fig: Quartet-03-different-kind}. It was shown in \cite{JS11} that a group of four five-element series-parallel networks can be used to realize such a function. Therefore, the number of configurations covering the case of regularity with $\mathds{R} < 0$ is reduced by one in the present paper.}
\end{remark}

\begin{remark}
The logical path of lemmas and theorems in this paper is as follows.
Theorem~\ref{theorem: final corollary} follows from Theorem~\ref{theorem: final condition} together with some results in the existing literature. Theorem~\ref{theorem: final condition} is the combination of Theorems~\ref{theorem: main theorem Res > 0} and \ref{theorem: main theorem Res < 0}, where Theorem~\ref{theorem: main theorem Res > 0} follows from Lemma~\ref{lemma: realization of Five-Element Bridge Networks with Two Reactive Elements of the Same Type} and Theorems~\ref{theorem: condition of the No. 85 network} and \ref{theorem: condition of the No. 86 network}, and Theorem~\ref{theorem: main theorem Res < 0} is derived  from Lemmas~\ref{lemma: type of elements}, \ref{lemma: all the possible quartets for one inductor and one capacitor}, and \ref{lemma: condition of the quartet 01} and Theorems~\ref{theorem: condition of the quartet 02} and \ref{theorem: condition of the quartet 03}.
The proof of Lemma~\ref{lemma: realization of Five-Element Bridge Networks with Two Reactive Elements of the Same Type} makes use of Lemma~\ref{lemma: topological structure}, the proof of Theorem~\ref{theorem: condition of the No. 85 network} makes use of Lemma~\ref{lemma: type of elements}, the proof of Theorem~\ref{theorem: condition of the No. 86 network} makes use of Lemmas~\ref{lemma: type of elements}, \ref{lemma: U V W lemma}, \ref{lemma: condition of the No. 86 network}, and \ref{lemma: condition of the No. 86' network}, and the proof of Theorem~\ref{theorem: condition of the quartet 02} makes use of Lemma~\ref{lemma: further condition of the No. 104 network}.
\end{remark}

\begin{remark}
The configurations of this paper can be connected to $RC$  low-pass filters.
By appropriately setting the values of components in the filters,
the frequency responses of the resulting networks can be adjusted, in order to better reject high-frequency noises and guarantee the low-frequency responses to approximate to those of the original configurations.
Applying passive network synthesis to the circuits with filter implementations needs to be further investigated.
\end{remark}

\section{Numerical Examples}    \label{sec: examples}

\begin{example}
{As shown in \cite{WHC12}, the function
\begin{equation}
Z_{e, 2}^{sy} = \frac{s^2 + 2.171\times 10^{8}s + 4.824\times 10^{9}}{1.632s^2 + 1.575\times 10^{8}s + 2.838\times 10^{8}}
\end{equation}
is the impedance of an external circuit in the machatronic suspension system, which optimizes the settling time at a certain velocity range and is realizable as a five-element series-parallel configuration in \cite[Fig.~18]{WHC12}. Since  $\mathds{R} > 0$, $\mathds{R} - 4a_0a_2b_0b_2 > 0$, and $\mathds{R} -  4a_2b_0(a_1b_1-2\mathds{B}) > 0$,
$Z_{e, 2}^{sy}$ satisfies the condition of Theorem~\ref{theorem: final condition}, so is realizable as a five-element bridge network.
Furthermore, $Z_{e, 2}^{sy}$ is realizable as Fig.~\ref{fig: Quartet-01-same-kind}(a) with $R_1 = 16.232~\Omega$, $R_2 = 0.637~\Omega$, $R_3 = 0.766~\Omega$, $C_1 = 0.0329~\text{F}$, and $C_1 = 1.411 \times 10^{-8}~\text{F}$, and  $Z_{e, 2}^{sy}$ is also realizable as Fig.~\ref{fig: Quartet-02-same-kind}(a) with $R_1 = 14.425~\Omega$, $R_2 = 1.378~\Omega$, $R_3 = 1.194~\Omega$, $C_1 = 4.157 \times 10^{-9}~\text{F}$, and $C_2 = 0.0355~\text{F}$. }
\end{example}

\begin{example}
{As shown in \cite{WCJS09}, the function
\begin{equation}
Z_{e, J_1}^{2nd} = \frac{1.665\times 10^5 s^2 + 5.776 \times 10^5 s + 5.466 \times 10^7}{s^2 + 1.544 \times 10^6 s + 0.342}
\end{equation}
is the impedance of an external circuit in the machatronic suspension system (LMIS3 layout), which optimizes $J_1$ (ride comfort) and is realizable as a five-element series-parallel configuration in \cite[Fig.~2(c)]{WCJS09}. It can be verified that $\mathds{R} < 0$ and
$Z_{e, J_1}^{2nd}$ is regular, implying that $Z_{e, J_1}^{2nd}$ is realizable as a five-element bridge network by Theorem~\ref{theorem: final condition}. Furthermore, $Z_{e, J_1}^{2nd}$ is realizable as Fig.~\ref{fig: Quartet-03-different-kind} with $R_1 = 1.598 \times 10^8~\Omega$, $R_2 = 0.374~\Omega$, $R_3 = 1.665 \times 10^5~\Omega$, $C_1 = 0.0282~\text{F}$, and $L_1 = 0.108~ \text{H}$.}
\end{example}

\section{Conclusion}    \label{sec: conclusion}
This paper has investigated the realization problem of biquadratic impedances as five-element bridge networks, where the biquadratic impedance was assumed to be not realizable with fewer than five elements.
Through investigating the realizability conditions of configurations  covering all the possible cases, a necessary and sufficient condition was derived for a biquadratic impedance to be realizable as a two-reactive five-element bridge network,  in terms of the coefficients only.
Through the discussions, a canonical form for biquadratic impedances was   utilized to simplify and combine the obtained conditions. Finally, a necessary and sufficient condition was obtained for the realizability of the biquadratic impedance as a general five-element bridge network.


\section*{Acknowledgment}

The authors are grateful to Professor Rudolf E Kalman for the enlightening discussion regarding this work.


\newpage

\begin{center}
{\LARGE Supplementary Material to: Realizations of Biquadratic Impedances as Five-Element Bridge Networks}
\end{center}







%
\IEEEpeerreviewmaketitle

\section*{I.~Introduction}
This report presents some supplementary material to  the paper entitled as ``Realizations of Biquadratic Impedances as Five-Element Bridge Networks'' \cite{CWLC17}.  For more background information of this field, refer to  \cite{Che07}--\cite{Smi02} and references therein.

\section*{II.~Definitions of the network duality}

Regardless of  the values of the elements, any one-port passive $RLC$ network $N$ can be regarded as a \textit{one-terminal-pair labeled graph} $\mathcal{N}$ with two distinguished \textit{terminal vertices} (see \cite[pg.~14]{SR61}), in which the labels designate passive circuit elements, namely resistors, capacitors, and inductors, which are labeled as $R_i$, $C_i$, and $L_i$, respectively.

Two natural maps acting on the labeled graph are defined as follows:
\begin{enumerate}
  \item $\text{GDu} :=$ Graph duality, which takes the one-terminal-pair graph into its dual (see \cite[Definition~3-12]{SR61}) while preserving the labeling.
  \item $\text{Inv}  :=$ Inversion, which preserves the graph but interchanges the reactive elements, that is, capacitors to inductors and inductors to capacitors with their labels $C_i$ to $L_i$ and $L_i$ to $C_i$.
\end{enumerate}
Consequently, one defines
\begin{equation*}
\text{Dual} :=  \text{network duality of one terminal-pair labeled graph} := \text{GDu} \circ \text{Inv} = \text{Inv} \circ \text{GDu}.
\end{equation*}

Consider a network $N$ whose one-terminal-pair labeled graph is $\mathcal{N}$. Denote $N^i$ as the network whose one-terminal-pair labeled graph is $\text{Inv}(\mathcal{N})$,  resistors are of the same values as those of $N$, and  inductors (resp. capacitors) are replaced by capacitors (resp. inductors) with reciprocal values. Denote $N^{id}$  as the network whose one-terminal-pair labeled graph is $\text{GDu}(\mathcal{N})$ and elements are of the reciprocal values to those of $N$. Denote $N^d$ as the network whose one-terminal-pair labeled graph is $\text{Dual}(\mathcal{N})$,   resistors  are of reciprocal values to those of $N$, and   inductors (resp. capacitors) are replaced by capacitors (resp. inductors) with same values.
Based on the mesh current and node voltage method, it can be proved that $Z(s)$ (resp. $Y(s)$) is realizable as the impedance (resp. admittance)
of a network $N$ whose one-terminal-pair labeled graph is $\mathcal{N}$,
if and only if $Z(s^{-1})$ (resp. $Y(s^{-1})$) is realizable as the impedance (resp. admittance) of $N^i$ whose one-terminal-pair labeled graph is $\text{Inv}(\mathcal{N})$, if and only if
$Z(s^{-1})$ (resp. $Y(s^{-1})$) is realizable as the admittance (resp. impedance) of $N^{id}$ whose one-terminal-pair labeled graph is $\text{GDu}(\mathcal{N})$, and if and only if
it is realizable as
the admittance (resp. impedance) of $N^d$ whose one-terminal-pair labeled graph is $\text{Dual}(\mathcal{N})$.


If a necessary and sufficient condition  is derived for  $H(s) =
\sum_{k=0}^m a_k s^k/\sum_{k=0}^m b_k s^k$ to be realizable as the impedance (resp. admittance) of the one-port network $N$ whose one-terminal-pair labeled graph is $\mathcal{N}$, then the corresponding condition for $N^i$ whose one-terminal-pair labeled
graph is $\text{Inv}(\mathcal{N})$ can be obtained from that for $N$ through conversion $a_k \leftrightarrow a_{m-k}$ and $b_k \leftrightarrow b_{m-k}$, $k \in \{0,1,...,\lfloor m/2 \rfloor \}$.
Consequently, the corresponding condition for $N^{id}$ whose one-terminal-pair labeled
graph is $\text{GDu}(\mathcal{N})$ can be obtained from that for $N$ through conversion $a_k \leftrightarrow b_{m-k}$, $k \in \{0,1,...,m\}$.
Furthermore, the corresponding condition for $N^d$  whose one-terminal-pair labeled graph is $\text{Dual}(\mathcal{N})$ can be obtained from that for $N$ through conversion $a_k \leftrightarrow b_k$,
$k \in \{0,1,...,m\}$
with the values of the elements for $N^d$ being obtained from those for $N$ through conversion
$R_i \rightarrow R_i^{-1}$, $C_i \rightarrow L_i$, $L_i \rightarrow C_i$, and
$a_k \leftrightarrow b_k$, $k \in \{0,1,...,m\}$.

\section*{III.~Proof of Lemma~4}

It follows from condition~(6) and the assumption of $W \neq 3$ that
$W > 3$. In order to simplify the proof of this lemma, one utilizes the coordinate  transformation
\begin{equation}  \label{eq: the coordinate revolution} \tag{III.1}
\begin{split}
U &= x \cos \frac{\pi}{4} - y \sin \frac{\pi}{4} = \frac{\sqrt{2}}{2} (x - y),  \\
V &= x \sin \frac{\pi}{4} + y \cos \frac{\pi}{4} =
\frac{\sqrt{2}}{2} (x + y),
\end{split}
\end{equation}
which does not affect the nature and proof of the lemma.
Thus, conditions~(5)--(7) are converted into the following:
\begin{align}
\Psi_1(x,y) &:= 2W^{-1}(W-1)^2 x^2
- 2W^{-1}(W+1)^2 y^2  - W^{-2}(W^2 - 1)^2   > 0,
\label{eq: equivalent condition 01} \tag{III.2} \\
\Psi_2(x,y) &:= 2W^{-1}(W+1)(W-3)x^2
- 2W^{-1}(W+3)(W-1)y^2 - W^{-2}(W^2-9)(W^2-1)  \geq 0,
\label{eq: equivalent condition 02} \tag{III.3}
\end{align}
and
\begin{equation} \label{eq: equivalent condition 03} \tag{III.4}
\begin{split}
\Sigma(x,y) :=& -4W^{-2}(W+3)(W+1)^3 y^4 + 8W^{-2}(W+3)(W+1)(W^2-2W-1)xy^3 \\
&+ 16W^{-1}(W^2-5)x^2y^2  -  8W^{-2}(W-1)(W-3)(W^2+2W-1)x^3y   \\
&+ 4W^{-2}(W-3)(W-1)^3x^4  -  2W^{-3}(W-1)(W+3)(W^2-3)(W+1)^2y^2 \\
&+  4W^{-3}(W^2-9)(W^2-1)^2xy  - 2W^{-3}(W-3)(W+1)(W^2-3) (W-1)^2x^2 > 0,
\end{split}
\end{equation}
where $x > 0$ and $W > 3$. Note that \eqref{eq: equivalent condition 01} and \eqref{eq: equivalent condition 02} are equivalent to
\begin{align}
x &> \frac{W+1}{W-1}\sqrt{\frac{2Wy^2 + (W-1)^2}{2W}} =: x_{\Psi_1}(y),
\label{eq: equivalent condition 01 x}  \tag{III.5}  \\
x &\geq \sqrt{\frac{(W+3)(W-1)(2Wy^2 + (W+1)(W-3))}{2W(W+1)(W-3)}} =: x_{\Psi_2}(y),
\label{eq: equivalent condition 02 x}  \tag{III.6}
\end{align}
respectively. Hence, it can be verified that
condition~\eqref{eq: equivalent condition 01 x} yields condition~\eqref{eq: equivalent condition 02 x} when
$y\in [-y_1, y_1]$, and condition~\eqref{eq: equivalent condition 02 x} yields   condition~\eqref{eq: equivalent condition 01 x} when
$y \in (-\infty, -y_1)\cup (y_1, +\infty)$, where
\begin{equation*}
y_1 = \frac{W-1}{2W}\sqrt{\frac{(W+1)(W-3)}{2}}.
\end{equation*}

Consider the case of $y \in [-y_1, y_1]$.
It suffices to show that condition~\eqref{eq: equivalent condition 01 x} implies condition~\eqref{eq: equivalent condition 03}. It can be verified that
\begin{equation}  \label{eq: partial diff Sigma} \tag{III.7}
\begin{split}
\Sigma_x(x,y) :=& \frac{\partial}{\partial x} \Sigma(x,y)  \\
=&  16W^{-2}(W-1)^3(W-3) x^3 -  24W^{-2}(W-1)(W-3)(W^2+2W-1)x^2y    \\
& +  32W^{-1}(W^2-5)x y^2  -  4W^{-3}(W-3)(W+1)(W^2-3)(W-1)^2 x   \\
& +  8W^{-2}(W+3)(W+1)(W^2-2W-1)y^3  + 4W^{-3}(W^2-1)^2(W^2-9) y,
\end{split}
\end{equation}
and
\begin{equation}  \label{eq: double partial diff Sigma} \tag{III.8}
\begin{split}
\Sigma_{xx}(x,y) :=& \frac{\partial^2}{\partial x^2} \Sigma(x,y)  \\
=&  48W^{-2}(W-1)^3(W-3) x^2  -  48W^{-2}(W-1)(W-3)(W^2+2W-1)xy    \\
&+  32W^{-1}(W^2-5)y^2
-  4W^{-3}(W+1)(W-1)^2(W^2-3)(W-3).
\end{split}
\end{equation}
Regarding \eqref{eq: double partial diff Sigma} as a quadratic function of $x$, its symmetric axis $x_s(y) = (W^2+2W-1)y/(2(W-1)^2)$ satisfies $x_s(y) < x_{\Psi_1}(y)$, as $4W(W-1)^4\left((x_s(y))^2 - (x_{\Psi_1}(y))^2\right) =
W(W^2+2W-1)^2y^2 - 2(W+1)^2(2Wy^2+(W-1)^2) = - W(3W^2+2W-3)(W^2-2W-1)y^2 - 2 (W+1)^2 (W-1)^4 < 0$.
Furthermore, one obtains
$\Sigma_{xx}(x_{\Psi_1}(y),y) = \Sigma_{xx1}(x_{\Psi_1}(y),y) -
\Sigma_{xx2}(x_{\Psi_1}(y),y)$,
where
\begin{align*}
\Sigma_{xx1}(x_{\Psi_1}(y),y) &= \frac{16(3W^4 - 4W^3 - 12W^2 - 4W + 9)y^2}{W^2} + \frac{4(W+1)(5W^2-3)(W-1)^2(W-3)}{W^3},  \\
\Sigma_{xx2}(x_{\Psi_1}(y),y) &= \frac{24\sqrt{2}  (W+1)(W-3)(W^2+2W-1)y}{W^2}\sqrt{\frac{2Wy^2+(W-1)^2}{W}}.
\end{align*}
It is obvious that $\Sigma_{xx1}(x_{\Psi_1}(y),y) > 0$. If
$y \in [-y_1, 0]$,  then $\Sigma_{xx2}(x_{\Psi_1}(y),y) \leq 0$, indicating  that $\Sigma_{xx}(x_{\Psi_1}(y),y) > 0$. If
$y \in (0, y_1]$,  then $\Sigma_{xx2}(x_{\Psi_1}(y),y) > 0$. One further obtains
$(\Sigma_{xx1}(x_{\Psi_1}(y),y)))^2 - (\Sigma_{xx2}(x_{\Psi_1}(y),y))^2 = - 2048W^{-3}(W^2-3W-2)(3W^4-2W^3-18W^2-8W+9)y^4 +  256W^{-4}(W+1)(W-3)(3W^5-19W^4+6W^3+86W^2+27W-39)(W-1)^2y^2 + 16W^{-6}(W-3)^2(W+1)^2(5W^2-3)^2(W-1)^4$,
which is positive for
$y \in (0, y_1]$. Hence, one concludes that $\Sigma_{xx}(x,y) > 0$ for $y \in [-y_1, y_1]$  and $x > x_{\Psi_1}(y)$, indicating that $\Sigma_{x}(x,y)$ increases  monotonically in $x$ when $x > x_{\Psi_1}(y)$ and $y$ is fixed for
$y \in [-y_1, y_1]$. For \eqref{eq: partial diff Sigma}, one has
$\Sigma_{x}(x_{\Psi_1}(y),y) = \Sigma_{x2}(x_{\Psi_1}(y),y)
- \Sigma_{x1}(x_{\Psi_1}(y),y)$,
where
\begin{equation*}
\Sigma_{x1}(x_{\Psi_1}(y),y) = \frac{16(W+1)(W^4-8W^2-8W+3)y^3}{W^2(W-1)} + \frac{8(W-1)(W-3)(W+2)(W+1)^2y}{W^2},
\end{equation*}
and
\begin{equation*}
\begin{split}
\Sigma_{x2}(x_{\Psi_1}(y),y) =&
\frac{2\sqrt{2}(W+1)}{W^3(W-1)}\sqrt{\frac{2Wy^2+(W-1)^2}{W}} \\
&\times \left( 4 W (W^4-4 W^2-8 W+3) y^2+(W-3) (W+1) (W^2+1) (W-1)^2 \right).
\end{split}
\end{equation*}
It can be verified that $\Sigma_{x2}(x_{\Psi_1}(y),y) > 0$. If
$y \in [-y_1, 0]$, then $\Sigma_{x1}(x_{\Psi_1}(y),y) \leq 0$, implying that $\Sigma_{x}(x_{\Psi_1}(y),y) > 0$. If
$y \in (0, y_1]$, then $\Sigma_{x1}(x_{\Psi_1}(y),y) > 0$. One obtains
\begin{equation} \label{eq: first sturm function} \tag{III.9}
\begin{split}
f(Y):=&(\Sigma_{x2}(x_{\Psi_1}(y),y))^2
- (\Sigma_{x1}(x_{\Psi_1}(y),y))^2    \\
=& \frac{2048(W+1)^2 S_1 Y^3}{W^2(W-1)^2}  -
\frac{256(W+1)^2S_2Y^2}{W^4}  \\
& +  16W^{-6}(W+1)^3(W-1)^2(W-3) S_3 Y
\\
& +  8W^{-7}(W+1)^4(W^2+1)^2(W-1)^4(W-3)^2,
\end{split}
\end{equation}
where $S_1  = W^4-6W^2-8W+3 > 0$, $S_2  = W^6-8W^5-5W^4+36W^3+63W^2-4W-3$, $S_3  = W^6-10W^5+15W^4+44W^3+39W^2-34W+9$, and $Y = y^2$.
The Sturm chain for  \eqref{eq: first sturm function} can be obtained through $f_0 (Y) = f(Y)$, $f_1 (Y) = f' (Y)$, $f_2 (Y) = - \text{rem} (f_0, f_1)$, and $f_3(Y) = -\text{rem} (f_1, f_2)$, where $\text{rem} (p_i, p_j)$   denotes the remainder of the polynomial long division of $p_i$ by $p_j$.
The sign of this chain at $Y = 0$ and $Y = Y_1 = y_1^2$ is as shown in Table~\ref{Signs of Sturm Chain} (the special case when $f_3(0) = f_3(Y_1) = -\infty$ is excluded, which does not affect the result), where $T = W^{12}+18W^{11}+18W^{10}-242W^9-685W^8+60W^7+3516W^6 +6260W^5+2407W^4-1422W^3+1234W^2-258W-27$ and $X = 3W^{10}+4W^9-47W^8-104W^7+54W^6 + 552W^5+1178W^4 +552W^3-1065W^2-1356W+549$, and
$V(Y)$ denote the number of sign variations   in the Sturm chain for \eqref{eq: first sturm function}. By investigating the roots of $S_3 = 0$, $T = 0$, and $X = 0$ in $W$ for $W > 3$, it is noted that $V(0) = V(Y_1)$. As shown in \cite[Chapter~XV]{Gan80}, the number of distinct roots of $f(Y)$ in $Y \in (0, Y_1)$  is $I_0^{Y_1} (f'(Y)/f(Y)) = V(0) - V(Y_1) = 0$. Together with $f(0) > 0$ and $f(Y_1) > 0$, it follows that $f(Y) > 0$ for $Y \in [0, Y_1]$, which further implies that $\Sigma_x (x_{\Psi_1}(y),y) > 0$ for $y \in [-y_1, y_1]$.
This means that $\Sigma(x,y)$ increases monotonically in $x > x_{\Psi_1}(y)$ for any $y \in [-y_1, y_1]$.  Substituting $x = x_{\Psi_1}(y)$ into $\Sigma(x,y)$ gives
$\Sigma(x_{\Psi_1}(y),y) = \Sigma_1(x_{\Psi_1}(y),y) - \Sigma_2(x_{\Psi_1}(y),y)$,
where

\begin{equation*}
\Sigma_1(x_{\Psi_1}(y),y) = -\frac{64(W+1)^2y^4}{W(W-1)^2}
+ \frac{8(W^2 - 4W - 3)(W+1)^2y^2}{W^3}
+ \frac{2(W+1)^3(W-1)^2(W-3)}{W^4},
\end{equation*}
and
\begin{equation*}
\Sigma_2(x_{\Psi_1}(y),y) =
\frac{4\sqrt{2}(W+1)^2y(-8W^2y^2 + (W+1)(W-3)(W-1)^2)}{W^3(W-1)^2}\sqrt{\frac{2Wy^2+(W-1)^2}{W}}.
\end{equation*}
It is also noted that $\Sigma_1(x_{\Psi_1}(y),y) \geq 0$ for
$y \in [-y_1, y_1]$.
Moreover, it is observed that if
$y \in (-y_1, 0)$, then $\Sigma_2(x_{\Psi_1}(y),y) \leq 0$, implying that $\Sigma(x_{\Psi_1}(y),y) \geq 0$. If
$y \in [0, y_1]$,   then one has
\begin{equation*}
\begin{split}
(\Sigma_1(x_{\Psi_1}(y),y))^2 - (\Sigma_2(x_{\Psi_1}(y),y))^2
=  4W^{-8}(W+1)^4(8W^2y^2 - (W+1)(W-3)(W-1)^2)^2
\geq 0.
\end{split}
\end{equation*}
Hence, $\Sigma(x_{\Psi_1}(y),y) \geq 0$ for $y \in [-y_1, y_1]$. As a result,  $\Sigma(x,y) > 0$ for
$y \in [-y_1, y_1]$ and $x > x_{\Psi_1}(y)$.

\begin{table}[!t]
\renewcommand{\arraystretch}{1.3}
\caption{The sign of the Sturm chain for \eqref{eq: first sturm function}.}
\label{Signs of Sturm Chain}
\centering
\begin{tabular}{|c||c|c|c|c||c|}
\hline $\text{sign}(f_i(Y))$ & $i=0$  &  $i=1$ & $i=2$ &  $i=3$ & Variations \\
\hline
$Y=0$ &  $+$ & $\text{sign}(S_3)$ &  $\text{sign}(-T)$ & $-$ & $V(0)$  \\
\hline
$Y=Y_1$ &$+$ & $+$ & $\text{sign}(-X)$ & $-$ & $V(Y_1)$ \\
\hline
\end{tabular}
\end{table}

Now, it remains to consider the case of $y \in (-\infty, -y_1)\cup(y_1, +\infty)$. In this case, one only needs to consider condition~\eqref{eq: equivalent condition 02 x}, that is, $x \geq x_{\Psi_2}(y)$, as it has been checked that $x_{\Psi_2}(y) > x_{\Psi_1}(y)$. It suffices to show that $\Sigma_{x}(x,y) > 0$ for $y \in (-\infty, -y_1)\cup(y_1, +\infty)$ and $x \geq x_{\Psi_2}(y)$, since it has been checked that $\Sigma(x_{\Psi_2}(y),y) > 0$ for $y \in (-\infty, -y_1)\cup(y_1, +\infty)$. These can be straightforwardly verified but the tedious detail is omitted for similarity of the presentation herein.

\ifCLASSOPTIONcaptionsoff
  \newpage
\fi

\end{document}